\documentclass[11pt]{article}
\usepackage[utf8]{inputenc}
\usepackage[percent]{overpic}
\usepackage{amsmath}
\usepackage{amssymb}
\usepackage{tikz}
\usepackage{tikz-cd}
\usepackage{mathtools}
\usepackage{stmaryrd}
\usepackage{amsthm}
\usepackage{array} % for centering text of given width
\usepackage[english]{babel}
\usepackage{enumitem}
\usepackage{xfrac}  % for special fractions
\usepackage{todonotes}
\usepackage[paper=a4paper, margin=2.cm]{geometry}
\usepackage{hyperref}
\usepackage{xcolor}

\hypersetup{
    unicode=false,          % non-Latin characters in bookmarks
    pdftoolbar=true,        % show toolbar?
    pdfmenubar=true,        % show menu?
    pdffitwindow=false,     % window fit to page when opened
    pdfstartview={FitH},    % fits the width of the page to the window
    pdftitle={Polyhedral models for K-theory of toric and flag varieties},    % title
    pdfauthor={ Monin, L.; Smirnov, E.},     % author
    pdfkeywords=  {Toric varieties} {Flag varieties} {Gelfand--Zetlin polytopes} {Gorenstein rings} % list of keywords
    pdfnewwindow=true,      % links in new window
    colorlinks=true,       % false: boxed links; true: colored links
    linkcolor=blue,          % color of internal links
    citecolor=red,        % color of links to bibliography
    filecolor=magenta,      % color of file links
    urlcolor=cyan           % color of external links
}

\theoremstyle{plain}
\newtheorem{theorem}{Theorem}[section]
\newtheorem{lemma}[theorem]{Lemma}
\newtheorem{proposition}[theorem]{Proposition}
\newtheorem{corollary}[theorem]{Corollary}

\theoremstyle{definition}
\newtheorem{definition}[theorem]{Definition}
\newtheorem{remark}[theorem]{Remark}
\newtheorem{example}[theorem]{Example}

\newcommand{\rleft}{\mathopen{}\mathclose\bgroup\left}
\newcommand{\rright}{\aftergroup\egroup\right}

\newcommand{\C}{{\mathbb{C}}}
\newcommand{\I}{{\mathbb{I}}}

\newcommand{\kk}{{\normalfont\mathbf{k}}}
\newcommand{\R}{{\mathbb{R}}}
\newcommand{\Z}{{\mathbb{Z}}}

\newcommand{\Q}{{\mathbb{Q}}}

\newcommand{\Bm}{{\mathcal{B}}}

\newcommand{\Dm}{{\mathcal{D}}}
\newcommand{\Fm}{{\mathcal{F}}}
\newcommand{\Vm}{{\mathcal{V}}}
\newcommand{\Um}{{\mathcal{U}}}
\newcommand{\Wm}{{\mathcal{W}}}
\newcommand{\Gm}{{\mathcal{G}}}

\newcommand{\Lm}{{\mathcal{L}}}

\newcommand{\Om}{{\mathcal{O}}}
\newcommand{\Pm}{{\mathcal{P}}}

\newcommand{\calL}{{\mathcal{L}}}

\newcommand{\kb}{{\mathbf{k}}}

\newcommand{\Ann}{\mathop{\mathrm{Ann}}}

\newcommand{\ch}{\mathop\mathrm{char}}

\newcommand{\Fl}{{\mathrm{Fl}}}
\newcommand{\GL}{{\mathrm{GL}}}

\newcommand{\Hom}{{\mathrm{Hom}}}
\newcommand{\ind}{{\mathrm{ind}}}

\newcommand{\Pic}{{\mathrm{Pic}}}

\newcommand{\pr}{{\mathrm{pr}}}
\newcommand{\pt}{{\mathrm{pt}}}

\newcommand{\Sym}{{\mathrm{Sym}}}
\newcommand{\vol}{{\mathrm{vol}}}

\newcommand{\Ehr}{{\mathrm{Ehr}}}
\newcommand{\Sh}{{\mathrm{Sh}}}
\newcommand{\Fun}{{\mathrm{Fun}}}
\usepackage{pgf,tikz}
\usepackage{mathrsfs}
\usetikzlibrary{arrows}
\definecolor{uuuuuu}{rgb}{0.26666666666666666,0.26666666666666666,0.26666666666666666}
\definecolor{yqyqyq}{rgb}{0.5019607843137255,0.5019607843137255,0.5019607843137255}
\definecolor{uququq}{rgb}{0.25098039215686274,0.25098039215686274,0.25098039215686274}

\title{Polyhedral models for $K$-theory of toric and flag varieties}
\author{Leonid Monin
  \and
Evgeny Smirnov
}
\date{\today}

\newcommand{\Addresses}{{% additional braces for segregating \footnotesize
  \bigskip
  \footnotesize

  \medskip

  L.~Monin, \textsc{EPFL, Lausanne, Switzerland}\par\nopagebreak
  \textit{E-mail address}: \texttt{leonid.monin@epfl.ch}
  
  \medskip
  
  E.~Smirnov, \textsc{HSE University, Russia}
  \par\nopagebreak
  \textit{E-mail address}: \texttt{esmirnov@hse.ru}
}}

\makeatletter
\newcommand{\subjclass}[2][1991]{%
  \let\@oldtitle\@title%
  \gdef\@title{\@oldtitle\footnotetext{#1 \emph{Mathematics subject classification.} #2}}%
}
\newcommand{\keywords}[1]{%
  \let\@@oldtitle\@title%
  \gdef\@title{\@@oldtitle\footnotetext{\emph{Key words and phrases.} #1.}}%
}
\makeatother

\subjclass[2020]{Primary 16E20, Secondary 14M15, 	14M25, 52B20}
\keywords{Toric varieties, flag varieties, Gelfand--Zetlin polytopes, Frobenius duality}

\theoremstyle{plain}
\newtheorem*{intro-thm}{Theorem}
\theoremstyle{definition}
\newtheorem*{intro-def}{Definition}

\begin{document}
\maketitle
\abstract{In 1992, Pukhlikov and Khovanskii  provided a description of the cohomology ring of toric variety as a quotient of the ring of differential operators on spaces of virtual polytopes. Later Kaveh generalized this construction to the case of cohomology rings for full flag varieties. 

In this paper we extend Pukhlikov--Khovanskii type presentation to the case of $K$-theory of toric and flag varieties. First, we study the Frobenius algebras obtained as quotients of the group algebra of free abelian group (possibly of infinite rank). Then we apply this construction to define a $K$-ring associated to a linear family of (virtual) polytopes. We study in detail two examples of such families: the family of integer (virtual) polytopes with a fixed normal fan and the family of (virtual) Gelfand-Zetlin polytopes. We show that the $K$-theory of toric and flag varieties can be realized as $K$-rings of the above families and use this to get natural set of relations in the above $K$-rings. Further, we describe the classes  of structure sheaves of toric orbit closures and Schubert varieties in type $A$ flag varieties. Finally, we show that our results also hold true in $T$-equivariant setting.}

%%%%%%%%%%%%%%%%%%%%%%%%%%%%%%%%%%%%%%%%%%%%%%%%%%%%%%%%%%%%%%%%%

%\tableofcontents

\section{Introduction}
\label{sec:intro}

%\subsection{Overview}

\noindent
This paper is devoted to the study of the $K$-theory of toric and flag
varieties by methods of polyhedral geometry. Our starting point is a classical idea going back to A.~Pukhlikov and A.~Khovanskii~\cite{PK92}, who described the cohomology ring of a smooth
complete toric variety $X$ as a quotient of a ring of differential
operators with constant coefficients by the annihilator of the volume
polynomial of the moment polytope of~$X$. This construction was
extended by K.~Kaveh~\cite{KavehVolume} to full flag varieties, with the
moment polytope replaced by a Gelfand--Zetlin polytope, and was used by
V.~Kiritchenko, the second-named author, and V.~Timorin~\cite{kirichenko2012schubert}
to build a polyhedral model for Schubert calculus on flag varieties.

At the heart of all these descriptions lies a form of \emph{Macaulay
duality}. Recall that for a finite-dimensional vector space $V$ over a
field $\kk$ of characteristic zero one identifies the symmetric algebra
$\Sym(V)$ with the ring of differential operators with constant
coefficients on $V$; the latter acts on polynomial functions on $V$.
For a polynomial $f$ one defines
\[
A_f=\Sym(V)/\Ann(f),
\qquad
\Ann(f)=\{\,D\in\Sym(V)\mid D\cdot f=0\,\}.
\]
One can show that, the algebra $A_f$ is Artinian, Gorenstein, and carries a
nondegenerate Frobenius pairing. Conversely, every Artinian Gorenstein
quotient of $\Sym(V)$ arises in this way. Moreover, if $f$ is a homogeneous polynomial, the algebra $A_f$ is graded the corresponding Frobenius pairing comes from Poincar\'e duality on $A_f$. Pukhlikov and Khovanskii observed that the cohomology ring of a smooth complete toric variety associated to a fan $\Sigma$ is isomorphic to the algebra $A_{\vol}$ associated to the volume polynomial on the linear family of (virtual) polytopes with normal fan $\Sigma$.

The purpose of this paper is to develop a $K$-theoretic analogue of this
theory and to apply it to toric and flag varieties. In the cohomology theory one
works with the volume, a homogeneous polynomial, and with differential
operators. For $K$-theory the natural replacement of the volume is the
number of lattice points, an inhomogeneous (Ehrhart) polynomial, and the
natural replacement of differentiation is \emph{finite differences}.
This leads us to replace $\Sym(V)$ by the group algebra of a lattice,
acting on functions on that lattice by shift operators. We carry out
this replacement in three steps: first an abstract duality theory for such ``discrete''
algebras (Section~\ref{sec:PKalgebras}); then a study of the resulting
rings for families of polytopes
(Sections~\ref{sec:polyprelim}--\ref{sec:polyKring}); and finally the
identification of these rings with the $K$-theory of toric and flag
varieties, together with explicit polyhedral representatives for
structure sheaves of orbit closures and Schubert varieties
(Sections~\ref{sec:polymodels}--\ref{sec:flags}). It worth to mention that unlike Puklikov--Khovanskii description for cohomology rings, our $K$-theoretic results work in $T$-equivariant setting (Theorems~\ref{thm:toricKeq} and~\ref{thm:kflag-string}). 

A part of the results of this paper was announced in the proceedings of FPSAC-2023~\cite{MoninSmirnov23}. A further generalization of this approach to an arbitrary oriented cohomology theory in the sense of Levine and Morel was given by V.~Petrov and G.~Shulga in~\cite{PetrovShulga24}; in this setting, the polynomial replacing the volume polynomial is computed using the Riemann--Roch theorem. Finally, in a forthcoming work \cite{dupraz}, motivated by the Pukhlikov--Khovanskii descriptions of the cohomology and $K$-rings authors develop Riemann--Roch theory for general pairs of discrete and continuous Frobenious quotients.

Throughout, by a \emph{lattice} we mean a free abelian group, possibly of
infinite rank.

% --------------------------------------------------------------
\subsection{Discrete Pukhlikov--Khovanskii rings}
% --------------------------------------------------------------

Let $\Lambda$ be a lattice and $\kk$ a field. Denote by
$\Fun(\Lambda,\kk)$ the space of all $\kk$-valued functions on
$\Lambda$, and by $\kk[\Lambda]$ the group algebra of $\Lambda$. The
group algebra acts on $\Fun(\Lambda,\kk)$ by translations: an element
$t=\sum_i a_i\lambda_i\in\kk[\Lambda]$ acts by
\[
(t\cdot f)(x)=\sum_i a_i\,f(x+\lambda_i),
\qquad x\in\Lambda .
\]
Under this action we identify $\kk[\Lambda]$ with the algebra
$\Sh(\Lambda)$ of \emph{shift operators with constant coefficients}.
For $x\in\Lambda$ we write $T_x$ for the shift by $-x$ and
$D_x=1-T_x$ for the corresponding \emph{difference operator}; these play
the role of the partial derivatives in the classical theory.

To any function we attach an algebra as follows.

\begin{intro-def}[Definition~\ref{def:discrete-pk}]
Let $f\in\Fun(\Lambda,\kk)$. Its \emph{annihilator} is the subspace
\[
\Ann(f)=\{\,s\in\Sh(\Lambda)\mid s\cdot f\equiv 0\,\},
\]
and the \emph{discrete Pukhlikov--Khovanskii ring}
(\emph{discrete PK-ring}) of $f$ is the quotient
\[
K_f=\Sh(\Lambda)/\Ann(f).
\]
\end{intro-def}

A linear functional $\ell$ on a commutative algebra $A$ defines a
symmetric bilinear pairing $\langle a,b\rangle_\ell=\ell(ab)$; we call
this pairing a \emph{Frobenius duality} when it is nondegenerate. Our
first result shows that the construction of Definition~\ref{def:discrete-pk}
always produces such a duality.

\begin{intro-thm}[{Theorem~\ref{thm:diffalg}}]
For any $f\in\Fun(\Lambda,\kk)$ the set $\Ann(f)$ is an ideal in
$\Sh(\Lambda)$, and the quotient $K_f$ carries a Frobenius duality
given by the functional
\[
\ell_f\colon K_f\to\kk,
\qquad
\ell_f([t])=(t\cdot f)(0).
\]
\end{intro-thm}

Conversely, the rings $K_f$ exhaust all commutative algebras with
Frobenius duality that are generated by invertible elements. %This is
%the discrete counterpart of the classification of Artinian Gorenstein
%quotients of $\Sym(V)$, and it is what makes the theory applicable to
%$K$-rings, which are generated by the (invertible) classes of line
%bundles.

\begin{intro-thm}[{Theorem~\ref{thm:goralgrec}}]
Let $A$ be a commutative algebra generated by invertible elements
$a_1,\dots,a_s\in A^\times$ and carrying a Frobenius duality given by a
functional $\ell$. Then $A\simeq K_f$, where
\[
f\colon\Z^s\to\kk,
\qquad
f(n_1,\dots,n_s)=\ell\bigl(a_1^{\,n_1}\cdots a_s^{\,n_s}\bigr).
\]
\end{intro-thm}

% Two features of this discrete theory deserve emphasis. First, no
% finiteness or grading hypotheses are needed: $\Lambda$ may have infinite
% rank and $f$ may be arbitrary. The ring $K_f$ is finite-dimensional
% exactly when $f$ is ``polynomial-like''; in particular, if $\Lambda$ has
% finite rank and $f$ is a polynomial, then $K_f$ is Artinian
% (Proposition~\ref{prop:polyart}). Second, the construction has a useful
% \emph{module} incarnation: $K_f$ is isomorphic, as a module over
% $\Sh(\Lambda)$, to the cyclic submodule $M_f=\Sh(\Lambda)\cdot f$ of
% $\Fun(\Lambda,\kk)$ generated by $f$ (Remark~\ref{rem:homology}). In the
% geometric applications $K_f$ plays the role of the $K$-ring (or
% cohomology ring) and $M_f$ the role of $K$-homology, the pairing between
% them being the Kronecker pairing. This module picture is what allows us,
% later, to make sense of certain ``virtual'' operators that do not
% literally exist in $K_f$ but do define elements of $M_f$.

Finally, the whole construction works with values in an arbitrary
abelian group $A$ in place of the field $\kk$
(Theorem~\ref{thm:dualitywithcoef}): to a function $f\colon\Lambda\to A$
one attaches a ring $K_f$ with an $A$-valued duality, and a homomorphism
$\phi\colon A\to B$ induces a homomorphism $K_f\to K_{\phi\circ f}$ of
the associated rings (Proposition~\ref{prop:basechange}).  Taking $A$ to
be the representation ring of a torus allows us to give a Pukhlikov--Khovanskii presentation for $T$-equivariant $K$-rings. In this setting, the homomorphism induced by the augmentation
$\sum k_i\lambda_i\mapsto\sum k_i$ is exactly the forgetful map from the $T$-equivariant to the ordinary $K$-rings.

% --------------------------------------------------------------
\subsection{Polytope $K$-rings}
% --------------------------------------------------------------

The functions $f$ relevant to geometry come from polytopes. We work with
\emph{virtual polytopes}: the group $\Pm$ of formal differences of
lattice polytopes under Minkowski addition, and, for a fixed fan
$\Sigma$, its subgroup $\Pm_\Sigma$ of virtual polytopes whose support
function is linear on the cones of $\Sigma$. The group $\Pm_\Sigma$ is a
lattice of finite rank. The basic invariant is the \emph{Ehrhart
polynomial} $\Ehr$, the unique polynomial function on $\Pm$ whose value
on an honest lattice polytope $P$ is the number $|P\cap\Z^d|$ of its
lattice points; its equivariant refinement is the \emph{integer point
transform} $\sigma(P)=\sum_{a\in P\cap\Z^d}\mathbf t^{\,a}$ and its
projections $\sigma_\pi$. (Virtual polytopes, convex chains, valuations,
the Ehrhart polynomial and the integer point transform are recalled in
Section~\ref{sec:polyprelim}.)

\begin{intro-def}[{Definition~\ref{def:K ring of linear family}}]%
\label{intro:def-linfamily}
A \emph{linear family of polytopes} parametrized by a lattice $\Lambda$
is a homomorphism $\lambda\mapsto P_\lambda$ from $\Lambda$ to $\Pm$,
i.e.\ $P_{\lambda+\mu}=P_\lambda+P_\mu$. Its \emph{$K$-ring} and
\emph{equivariant $K$-ring} are the discrete PK-rings of its Ehrhart
polynomial and its integer point transform respectively:
\[
K_\phi=\Sh(\Lambda)/\Ann\bigl(\Ehr(P_\lambda)\bigr),
\qquad
K_\phi^T=\Sh(\Lambda)/\Ann\bigl(\sigma(P_\lambda)\bigr).
\]
\end{intro-def}

The main computational result of this part of the paper is an explicit
presentation, by generators and relations, of the $K$-ring $K_\Sigma$ of
the tautological family attached to a smooth complete fan~$\Sigma$. We
stress that $\Sigma$ need not be projective: $K_\Sigma$ is defined
whether or not the family $\Pm_\Sigma$ contains an honest convex
polytope.

\begin{intro-thm}[{Theorem~\ref{thm:polalgrel}}]\label{intro:thm-polalgrel}
Let $\Sigma\subset(\R^n)^\vee$ be a smooth complete fan with rays
$\rho_1,\dots,\rho_r$ and primitive ray generators $e_1,\dots,e_r$.
Then
\[
K_\Sigma\simeq
\Q[T_1^{\pm1},\dots,T_r^{\pm1}]\big/(I+J),
\]
where $I$ is generated by the products
$(1-T_{i_1}^{-1})\cdots(1-T_{i_t}^{-1})$ over the collections of distinct
rays $\rho_{i_1},\dots,\rho_{i_t}$ that do \emph{not} span a cone of
$\Sigma$, and
\[
J=\Bigl\langle\,\textstyle\prod_{i=1}^r T_i^{\langle u,e_i\rangle}-1
\ \colon\ u\in\Z^n\,\Bigr\rangle .
\]
\end{intro-thm}

The two ideals have transparent meaning: the relations $J$ encode
translation invariance of the Ehrhart polynomial, while the relations
$I$ express that a product of difference operators
$D_i=1-T_i^{-1}$ over rays not forming a cone annihilates $\Ehr$. In
fact, for any cone $\sigma$ spanned by $\rho_{i_1},\dots,\rho_{i_t}$, the
\emph{face monomial} $D_\sigma=D_{i_1}\cdots D_{i_t}$ acts on $\Ehr(P)$
by restricting to the face of $P$ dual to $\sigma$
(Lemma~\ref{lem:Dsig}); this is the discrete analogue of the fact that
in cohomology a product of divisor classes computes the volume of a
face. The proof of Theorem~\ref{thm:polalgrel} occupies
Section~\ref{sec:polyKring}: after checking that the relations hold, we
produce an explicit additive basis of \emph{separatrix monomials},
indexed by the maximal cones of $\Sigma$, by means of a Morse-theoretic
argument with respect to a generic linear function. The argument
follows the Morse-theoretic method introduced by Timorin
in~\cite{timorin1999analogue} (see also~\cite{khovanskii1986hyperplane}),
adapted here to the $K$-theoretic setting.

% --------------------------------------------------------------
\subsection{Polyhedral models for $K$-theory of toric and flag
varieties}
% --------------------------------------------------------------

We now identify the abstract polytope $K$-rings with honest $K$-theory.
The mechanism is uniform: for a smooth variety $X$ with a cell
decomposition, the holomorphic Euler characteristic
$\langle\Fm,\Gm\rangle=\chi(X,\Fm\otimes\Gm)$ is a Frobenius duality on
$K_0(X)\otimes\Q$, and if $K_0(X)$ is generated by line bundles then
Theorem~\ref{thm:goralgrec} presents it as a discrete PK-ring. It
remains to identify the relevant function, which in each case is an
Euler characteristic computed combinatorially.

For toric varieties the Euler characteristic of a line bundle is its
Ehrhart polynomial, and we obtain the following.

\begin{intro-thm}[{Theorem~\ref{thm:toricK}}]\label{intro:thm-toricK}
Let $\Sigma$ be a smooth complete fan and $Y_\Sigma$ the corresponding
toric variety. Then
\[
K_0(Y_\Sigma)\simeq K_\Sigma=\Sh(\Pm_\Sigma)/\Ann(\Ehr).
\]
\end{intro-thm}

Combined with Theorem~\ref{intro:thm-polalgrel}, this recovers the
classical generators-and-relations description of $K_0(Y_\Sigma)$
(Corollary~\ref{cor:Krelations}; cf.~\cite{borisov2005k, sankaran08}) by
a purely polyhedral argument. Under the isomorphism, the structure sheaf
of the orbit closure $X_\sigma$ is represented by the face monomial
$D_\sigma$ (Proposition~\ref{prop:eqcoh} and the following statement).
The equivariant theory is obtained verbatim, with $\Ehr$ replaced by the
integer point transform $\sigma(P)=\sum_{u\in P\cap M}e^u$
(Theorem~\ref{thm:toricKeq}), and the nonequivariant statement is
recovered from it by the augmentation map of
Proposition~\ref{prop:basechange}.

For flag varieties the role of the polytope is played by string
polytopes. Let $G$ be a connected reductive group, $B$ a Borel subgroup,
$T\subset B$ a maximal torus, and $\Lambda=X^*(T)$ the weight lattice.
By the Borel--Weil--Bott theorem the $T$-equivariant Euler
characteristic of the line bundle $\calL_\lambda$ on $G/B$ is, up to
sign, the character of an irreducible representation, which we denote by $F^T_G(\lambda)$, and its nonequivariant specialization is the \emph{Weyl polynomial}
$F_G(\lambda)=\dim V(\lambda)$, a polynomial of degree $\dim(G/B)$ on
$\Lambda$. Already this yields a presentation of the $K$-ring
(Theorem~\ref{thm:kflag-general}):
\[
K_0(G/B)=\Sh(\Lambda)/\Ann(F_G),
\qquad
K_0^T(G/B)=\Sh(\Lambda)/\Ann(F_G^T).
\]
To make this polyhedral we use that, for a dominant weight $\lambda$,
the Weyl polynomial is the Ehrhart polynomial of a \emph{string
polytope} $S(\lambda)$ (a Gelfand--Zetlin polytope in type~$A$), and the
character of $V(\lambda)$ is a projected integer point transform of
$S(\lambda)$, which we denote by $\sigma_\pi(S(\lambda))$. String polytopes are only piecewise-linear in $\lambda$,
but on each cone of a suitable fan on the dominant chamber they form a
linear family, which we extend to a virtual linear family on all of
$\Lambda$. A key technical point (Proposition~\ref{prop:virtual
equality}), proved using Brion's theorem to propagate the Weyl character
formula off the dominant cone, is that the resulting identities
\[
\chi(G/B,\Lm_\lambda)=\Ehr(S(\lambda)),
\qquad
\chi_T(G/B,\Lm_\lambda)=\sigma_\pi(S(\lambda))
\]
remain valid for \emph{all} $\lambda\in\Lambda$. This gives the
polyhedral model.

\begin{intro-thm}[{Theorem~\ref{thm:kflag-string}}]\label{intro:thm-flag}
Let $G/B$ be a generalized full flag variety, $\Lambda=X^*(T)$, and
$S(\lambda)$ the (possibly virtual) string polytope of $\lambda$. Then
\[
K_0(G/B)=\Sh(\Lambda)/\Ann\bigl(\Ehr(S(\lambda))\bigr),
\qquad
K_0^T(G/B)=\Sh(\Lambda)/\Ann\bigl(\sigma_\pi(S(\lambda))\bigr).
\]
\end{intro-thm}

% --------------------------------------------------------------
\subsection{Schubert classes and $K$-theoretic Schubert calculus}
% --------------------------------------------------------------

The last part of the paper, Sections~\ref{sec:gz}--\ref{sec:flags},
makes the flag model fully explicit for $G=\GL(n)$, where $S(\lambda)$ is
the Gelfand--Zetlin polytope $GZ(\lambda)$ and $K_{GZ}$ denotes the
corresponding polytope $K$-ring. We first describe the polytope ring of
$GZ(\lambda)$ combinatorially. Since $GZ(\lambda)$ is not simple, the
naive face relations need to be organized carefully; we show that they
are generated by an explicit family of \emph{six-term relations} among
the faces.

\begin{intro-thm}[{Theorem~\ref{thm:6term}}]\label{intro:thm-6term}
In the polytope ring of $GZ(\lambda)$ the following relations hold for
all admissible $(r,s)$:
\[
[F^{rs}]+[F_{r+1,s}]-[F^{rs}_{r+1,s}]
=[F_{rs}]+[F^{r+1,s-1}]-[F^{r+1,s-1}_{rs}],
\]
where summands indexed out of range are set to zero.
\end{intro-thm}

Their linear parts recover the relations in the cohomological
Pukhlikov--Khovanskii ring of~\cite{kirichenko2012schubert}, so the
six-term relations are a genuine $K$-theoretic refinement.

Our main result on Schubert calculus identifies the structure sheaves of
Schubert varieties in $K_{GZ}$. To a permutation $w\in S_n$ one
associates, following~\cite{kirichenko2012schubert}, a collection
$\Fm(w)$ of \emph{Kogan faces} of $GZ(\lambda)$. We then set
\[
\Dm_w=\sum_{\Gamma\in\Fm(w)}(-1)^{\ell(w)-\dim\Gamma}\,D_\Gamma,
\]
where $D_\Gamma$ is the operator computing the lattice-point count of
the face $\Gamma$.

\begin{intro-thm}[{Theorem~\ref{thm:schubert classes}}]\label{intro:thm-schubert}
For $w\in S_n$, the class $[\Om_w]$ of the structure sheaf of the
Schubert variety $X_w\subset\GL(n)/B$ is represented in $K_{GZ}$ by
$\Dm_w$. In particular, $\Dm_w$ is a well-defined element of $K_{GZ}$.
\end{intro-thm}

A phenomenon with no toric counterpart appears here: since the
Gelfand--Zetlin fan is singular, the individual face operators
$D_\Gamma$ need \emph{not} exist in $K_{GZ}$. The point of
Theorem~\ref{thm:schubert classes} is that the alternating sum $\Dm_w$ is
nevertheless well defined. 
%a fact most naturally explained through the module $M_f$ of Remark~\ref{rem:homology}, or, equivalently, by passing to a smooth subdivision of the Gelfand--Zetlin fan as inProposition~\ref{prop:extended}. 
This result provides a polyhedral model
for $K$-theoretic Schubert calculus on $\GL(n)/B$ and lifts to
$K$-theory the description of Schubert cohomology classes obtained in
type~$A$ in~\cite{kirichenko2012schubert}. We illustrate the calculus by computing several products of
structure sheaves of Schubert varieties in $K(\GL(3)/B)$
(Subsection~\ref{ssec:example_gl3}).

The analogous statements in other types should follow from Littelmann's string polytopes, which we leave to
future work. As a first step, the polyhedra description of cohomology classes of Schubert varieties in type~$C$ was obtained by
N.~Fujita~\cite{Fujita22}.

\subsection{Acknowledgements}
This project started during the authors' visit to the Weizmann Institute of Science in June 2022. We thank this institution and especially Dmitry Gourevich and Dmitry Novikov for their warm hospitality. We also express our deep gratitude to Michel Brion and Igor Makhlin for fruitful discussions.

LM was partially supported by the DFG project number 539974215 and by SNSF grant 224099. ES was partially supported by the Basic Research Program of HSE University (HSE-BR-2025-84).
%%%%%%%%%%%%%%%%%%%%%%%%%%%%%%%%%%%%%%%%%%%%%%%%%% %%%%%%%%%%%%%%%

\section{Discrete Pukhlikov--Khovanskii algebras}\label{sec:PKalgebras}

In this section we define discrete version of Pukhlikov--Khovanskii algebras and prove some basic results about them. 
%A more comprehensive theory of Pukhlikov--Khovanskii algebras will be developed in forthcoming paper \cite{dupraz}. 

\subsection{Algebra associated to a function on a lattice}\label{ssec:alg}
%Here we associate an algebra with Frobenius duality to any function $f\colon \Lambda\to \kk$ on a lattice. 

Let $\kk$ be any field and let $A$ be a commutative algebra with identity over $\kk$. In further subsections of the paper we will stick with $\kk= \Q$; however, here we  work with a more general setting. A linear function $\ell \colon A\to \kk$ defines a symmetric bilinear pairing on $A$ via:
\[
\langle a,b\rangle_{\ell} := \ell(a\cdot b) \text{ for any } a, b\in A. 
\]
\begin{definition}
  A pairing $\langle\cdot,\cdot\rangle_\ell$ on algebra $A$ is called \emph{Frobenius duality} if it is non-degenerate. 
\end{definition}
The main objective of this subsection is to give a construction of algebras with Frobenius duality from a $\kk$-valued function on a lattice. Recall that by a lattice we mean a (possibly infinitely generated) free abelian group. 

Let $\Lambda$ be a lattice and let $\kk$ be any field. We will denote by $\Fun(\Lambda,\kk)$ the set of maps $f\colon\Lambda\to \kk$.
%(note that $\Fun(\Lambda,\kk)$ has an algebra structure given by pointvise addition and multiplication of functions). 
Further, denote  the group algebra of $\Lambda$ by $\kk[\Lambda]$. The group algebra $\kk[\Lambda]$ is acting on the set of functions $\Fun(\Lambda,\kk)$ via shift operators. That is, for $t = \sum_{i=1}^r a_i\lambda_i\in \kk[\Lambda]$ and $f\in \Fun(\Lambda,\kk)$ we have
\[
t\cdot f(x) = \sum_{i=1}^r a_i f(x+\lambda_i),\text{ for any } x\in \Lambda.
\]
In what follows we identify the group algebra $\kk[\Lambda]$ with the algebra of shift operators with constant coefficients $\Sh(\Lambda)$ on $\Lambda$. We will further denote by  $T_x$  the shift operator by $-x$ (note the sign choice!) and by $D_x=1-T_x$ the corresponding difference operator  for any $x\in \Lambda$.

\begin{definition}\label{def:discrete-pk}
   Let $f\in \Fun(\Lambda,\kk)$ be any function, and let $\Ann(f) = \{s\in \Sh(\Lambda) \,|\, s\cdot f \equiv 0\}$ be its annihilator in the ring of shift operators with constant coefficients. Then we will call the quotient algebra
   $K_f := \Sh(\Lambda)/\Ann(f)$  the \emph{discrete Pukhlikov--Khovanskii algebra} associated to $f$.
\end{definition}

\begin{theorem}\label{thm:diffalg}
Let $f\in \Fun(\Lambda,\kk)$ be any function. Then the discrete Pukhlikov--Khovanskii algebra $K_f$ is well-defined, i.e., $\Ann(f)$ is an ideal in $\Sh(\Lambda)$. Moreover, $K_f$ has a Frobenius duality defined by a function
\[
\ell_f\colon \Sh(\Lambda)\to \kk,\quad \ell_f\colon t \mapsto t\cdot f(0).
\]
In particular, linear function $\ell_f$ descends to a well-defined linear function on $K_f$.
\end{theorem}

\begin{proof}
Since the action of $\Sh(\Lambda)$ on $\Fun(\Lambda,\kk)$ is linear, for checking that $\Ann(f)$ is an ideal it is enough to show that $s s' \in \Ann(f)$ whenever $s\in \Ann(f)$. Indeed, $(s's)\cdot f = s'\cdot(s\cdot f)\equiv 0$.

Then note that $\ell_f(s)=0$ whenever $s\in\Ann(f)$, hence $\ell_f$ descends to a well-defined linear function on $K_f$. To show that $\ell_f\colon K_f\to \kk$ defines a Frobenius duality on $K_f$ it is enough to show that for any $s\in \Sh(\Lambda)$ such that $s\notin \Ann(f)$, there exists $s'$ with $\ell_f(ss')\ne 0$. For this notice that if $s\cdot f\not\equiv 0$, there exist $x\in \Lambda$ such that $s\cdot f(x)\ne 0$. Then we can take $s'= S_{x}$ as $\ell_f(S_x s)=S_x\cdot( s\cdot f)(0)= s\cdot f(x) \ne 0$.
\end{proof}

The construction from Theorem~\ref{thm:diffalg} is quite general, as the following theorem shows.

\begin{theorem}\label{thm:goralgrec}
Let $A$ be a commutative algebra generated by invertible elements $a_1,\ldots, a_s \in A^\times$ with Frobenius duality given by a function $\ell\colon A\to \kb$. Then $A\simeq K_f$ with 
\[
f\colon \Z^s\to \kb,\quad f\colon (n_1,\dots,n_s)\mapsto \ell(a_1^{n_1}\ldots a_s^{n_s}).
\]
\end{theorem}
\begin{proof}
    Indeed, since $a_1,\ldots, a_s$ generate the algebra $A$, there is a surjective homomorphism 
    \[
    \phi\colon \Sh(\Z^s)\simeq \kb[T_1^{\pm 1},\ldots T_s^{\pm 1}] \to A,\quad T_i^{\pm 1}\mapsto a_i^{\pm 1}. 
    \]
    Hence, it is enough to show that $\ker(\phi)=\Ann(f)$. For this, we notice that by definition of $f$, we get $s\cdot f(0)=\ell(\phi(s))$ for any $s\in \Sh(\Z^s)$. In particular, we get that the pullback of $\langle\cdot,\cdot\rangle_\ell$ to $\Sh(\Z^s)$  is a pairing defined by
    \[
    \langle s, t \rangle_f = s\cdot t\cdot f(0)\text{ for } s,t\in \Sh(\Z^s).
    \]
    Moreover, since $\langle\cdot,\cdot\rangle_\ell$ is non-degenerate on $A$, we have $\ker(\phi)=\ker\langle \cdot, \cdot \rangle_f$. The proof follows from the fact that $\ker\langle \cdot, \cdot \rangle_f = \Ann(f)$.
\end{proof}

\begin{remark}
A version of Theorem~\ref{thm:diffalg} is also true for any numerical semigroup $\Lambda$ \cite{dupraz}. In this more general form Theorem~\ref{thm:diffalg} is closely related to the explicit version of Macaulay duality recently studied in \cite{khovanskii2021gorenstein} (see also \cite{hof2020, uspehi} for related results).
\end{remark}

%Note that for a general function $f\colon\Lambda\to \kk$ the ideal $\Ann(f)$ might be trivial. 

This construction makes sense for an arbitrary function $f$. However, the generic situation is not interesting: for ``most'' functions $f$, we have $\Ann(f)=0$ (see Lemma~\ref{lem:triv} below).
%In fact, as illustrated by Lemma~\ref{lem:triv}, for most functions $f$, we have $\Ann(f)=0$. 
However, for some classes of functions $f$, the ideal $\Ann(f)$ is nontrivial.
For instance, if $f$ is a quasipolynomial on the lattice $\Lambda$, the quotient algebra $K_f$ is Artinian, i.e. is a finite dimensional vector space over~$\kk$.

\begin{lemma}\label{lem:triv} Let $\Lambda =\Z$ be a one-dimensional lattice, and let $T=\sum_{i=r}^s \lambda_iT_i$ for some $r,s\in\Z$. Then the set of functions $f\colon\Lambda\to \kk$ annihilated by $T$ is a finite-dimensional vector subspace in the space of quasipolynomials. 
\end{lemma}

\begin{proof}
The functions annihilated by $T$ are exactly those satisfying the linear recurrence relation
\[
\sum_{i=r}^s \lambda_i f(x-i)=0,
\]
and the space of such functions has dimension $s-r+1$: it is spanned by quasipolynomials $x^je^{\alpha}$, where $\alpha$ is a root of the characteristic equation for this recurrence relation, and $j$ is strictly less than the multiplicity of this root.
\end{proof}

This shows that for any function $g$ that is not a quasipolynomial on $\Z$ the annihilator $\Ann(g)$ is trivial. Note that for higher dimensions this is not true: there exist non-quasipolynomial functions with nonzero annihilator. Further we will be mostly interested in the case of polynomial $g$.

% \begin{example}\label{ex:triv}
% Let $g(x) = e^{e^{x}}$ be a double exponential function on one-dimensional lattice $\Z$ and let $T=\sum_{i=1}^s\lambda_iT_{k_i}$ with $k_1<\ldots<k_s$. Then $|T\cdot g(x)|> 0$ for $x\gg 0$, in particular $T\cdot g\not\equiv 0$  so $\Ann(g)$ is trivial.  
% \end{example}

% \begin{lemma}\label{lem:triv}
% Let $\Lambda =\Z$ be a one-dimensional lattice. Then for each nontrivial shift operator $T=\sum_{i=1}^k\lambda_it_{r_i}$ the set of functions $f\colon\Z\to \kk$ annihilated by $T$ is a finite-dimensional vector space. 

% %In particular, if $\kk$ is a countable, the set of functions $f$ with nontrivial annihilator $\Ann(f)\ne 0$ is countable, while the set of all functions $\kk^\Z$ is uncountable.
% \end{lemma}
% \begin{proof}
%     Since $T=\sum_{i=1}^s \lambda_iT_{k_i}$ with $k_1<\ldots<k_s$ is a linear operator, the set of functions annihilated by $T$ is a vector space. Moreover, if the function $f:\Z\to \kk$ is annihilated by $T$ it is determined by its values at points $0,\ldots, k_s-k_1$, hence the annihilator of $T$ is a finite dimensional vector space. The second statement obviously follows.
% \end{proof}

\begin{proposition}\label{prop:polyart}
  Let $\Lambda\simeq \Z^r$ be of finite rank, and let $f$ be a polynomial function on $\Lambda$. Then $K_f$ is an Artinian algebra.
\end{proposition}

\begin{proof}
Let $D^+_i= 1- T_i$ and $D_i^-=1- T_i^{-1}$ be the standard difference operators for $i=1,\ldots,r$. Then  monomials in $D_i^+$ and $D_j^-$ additively generate $\kk[\Lambda]$. The statement follows from the fact that for a polynomial $f$ of degree $d$, one has $(D^{\pm}_1)^{k_1}\ldots (D^{\pm}_r)^{k_r}\cdot f=0$ for $|k_1|+\ldots+|k_r|>d$.
\end{proof}

%\begin{colortext}{teal}
    
We finish this subsection with a description of the relation between two algebras $K_f, K_g$ defined by pairs $(\Lambda_g,g)$ and $(\Lambda_f,f)$ such that there is a lattice homomorphism $\sigma\colon \Lambda_g\to \Lambda_f$ with $g=\sigma^* f$. Our description is parallel to \cite[Proposition~2.4]{kirichenko2012schubert}.

%\todoillm{Replace by Poincar\'e duality description?}

\begin{proposition}\label{prop:extended}
  There exists an abelian group $M_{f,g}$ with an epimorphism $\pi\colon K_f\to M_{f,g}$ and a monomorphism $\iota\colon K_g\to M_{f,g}$ such that
   $\pi(\widetilde \alpha \widetilde \beta) =\iota(\alpha \beta)$  whenever $\pi(\widetilde \alpha)=\iota(\alpha)$ and $\pi(\widetilde \beta)=\iota(\beta)$.
\end{proposition}
\begin{proof}
We define a group $M_{f,g}$ to be the quotient $\Sh(\Lambda_f)/I_{f,g}$ of $\Sh(\Lambda_f)$ by subgroup $I_{f,g}$, where 
\[
I_{f,g} = \{s \in \Sh(\Lambda_f)\,|\, \sigma^* (s\cdot f) \equiv 0\}.
\]
The ideal $\Ann(f)$ is clearly contained in $I_{f,g}$, hence we obtain an epimorphism $\pi\colon K_f\to M_{f,g}$. 

Now, let $s\in \Sh(\Lambda_g) $ be a shift operator and let $[s]$ be its class in $K_g$, we define the injection $\iota\colon K_g\to M_{f,g}$ via
\[
\iota([s]) = \sigma_*(s),
\]
where $\sigma_*\colon\Sh(\Lambda_g)\to \Sh(\Lambda_f)$ is a natural homomorphism of group algebras extending map of lattices $\sigma\colon\Lambda_g\to \Lambda_f$. The homomorphism $\iota$ is well defined since for every shift operator $s\in \Sh(\Lambda_g)$ and a function $h\colon\Lambda_f\to \kk$ we have the following identity:
\[
\sigma^*(\sigma_*(s)\cdot h)=s\cdot(\sigma^*\cdot h).
\]
Indeed, if $s\in \Ann(g)$, we get
\[
\sigma^*(\sigma_*(s) \cdot f) = s\cdot (\sigma^*\cdot f)=s\cdot g\equiv 0
\]
and hence $s\in I_{f,g}$, so $\iota([s])=0$. By a similar argument, $\iota$ is also injective. 

Finally, the equality $\pi(\widetilde \alpha \widetilde \beta) =\iota(\alpha \beta)$  for $\pi(\widetilde \alpha)=\iota(\alpha)$ and $\pi(\widetilde \beta)=\iota(\beta)$ follows directly from the fact that $\sigma_*\colon\Sh(\Lambda_g)\to \Sh(\Lambda_f)$ is a homomorphism of algebras.
\end{proof}

 \begin{remark}
 It is an interesting  question for which pairs $(\Lambda_g,g)$ and $(\Lambda_f,f)$ as above there exists a ring homomorphism $K_g \to K_f$ induced by a lattice map $\Lambda_g\to \Lambda_f$. A natural criterion which guaranties the existence of the ring homomorphism  $K_g \to K_f$ would allow to study the categories of discrete and continuous PK-algebras. This is especially interesting in the context of Riemann--Roch theory for PK-algebras. See \cite[Section 5.2]{dupraz} for examples and more detail.
 \end{remark}

\begin{remark}\label{rem:homology}
    Note that two elements $s,t\in \Sh(\Lambda)$ give  the same elements in the quotient $K_f= \Sh(\Lambda)/\Ann(f)$ if and only if 
    \[
    s\cdot f(\lambda) = t\cdot f(\lambda) \quad \text{for every } \lambda\in \Lambda.
    \]
    Therefore, one can think of elements of $K_f$ as functions on the lattice $\Lambda$ which can be obtained from $f$ by applying difference operators.

    In other words, the algebra $K_f$ is isomorphic as a module over $\Sh(\Lambda)$ to the cyclic submodule $M_f=\Sh(\Lambda)\cdot f$ of $\Fun(\Lambda, \kk)$ where the isomorphism is given by
    \[
    K_f\ni [s] \mapsto s\cdot f\in M_f.
    \]
    The same holds for the continuous version of the above construction. 
    In the applications to the computation of the cohomology rings or $K$-rings of algebraic varieties this can be interpreted as the Kronecker pairing. Indeed, the ring $K_f$ plays the role of a cohomology ring, and module $M_f$ plays the role of a homology group with the Kroneker pairing given by the natural action of $K_f$ on~$M_f$.
\end{remark}

\subsection{Rings with duality with different coefficients}\label{ssec:dualitywithcoefficients}
In this subsection we give a more general construction of algebras with duality with coefficients in a general abelian group. We will use this construction to reconstruct torus equivariant $K$-theory of smooth complete toric varieties in Subsection~\ref{ssec:eqK}.

Let $A$ be an abelian group, and let $R$ be a commutative ring. We will say that $\langle\cdot,\cdot\rangle$ is a \emph{non-degenerate pairing} on $R$ if for any $0\ne x\in R$ there is $y\in R$ such that $\langle x,y\rangle \ne 0$.
Similarly to Theorem~\ref{thm:diffalg} and~\ref{thm:goralgrec} we get the following results.
\begin{theorem}\label{thm:dualitywithcoef}
    Let $f\colon\Lambda \to A$ be any function. Then
$\Ann(f) = \{s\in \Sh(\Lambda) \,|\, s\cdot f \equiv 0\}$
is an ideal in $\Sh(\Lambda)$, and the quotient algebra $K_f := \Sh(\Lambda)/\Ann(f)$ has a Frobenius duality defined by a function
\[
\ell_f\colon \Sh(\Lambda)\to \kk,\quad \ell_f\colon t \mapsto t\cdot f(0).
\]
Moreover, if $R$ is a commutative ring generated by invertible elements $a_1,\ldots, a_s \in R^\times$ with Frobenius duality given by a function $\ell\colon R\to A$, then $R\simeq K_f$ with 
\[
f\colon \Z^s\to \kb,\quad f\colon (n_1,\dots,n_s)\mapsto \ell(a_1^{n_1}\ldots a_s^{n_s}).
\]
\end{theorem}

\begin{proposition}\label{prop:basechange}
 Let $A,B$ be two abelian groups and let $\phi:A\to B$ be a group homomorphism. Finally, let $f\colon\Lambda\to A$ be any function, and let $K_f,K_{\phi\circ f}$ be corresponding rings with duality valued in $A$ and $B$ respectively. Then there exists a group homomorphism
 \[
  \pi\colon K_f\to K_{\phi\circ f} \text{ such that } \phi\left(\langle x, y \rangle_f\right) = \langle \pi(x),\pi(y)\rangle_{\phi\circ f} 
 \]
 for every $x,y\in R_f$.
\end{proposition}

We will use Proposition~\ref{prop:basechange} for the computation of $T$-equivariant cohomology of toric and flag varieties. In particular, in that case we have $A=K_0^T(\pt) = \Z[M]$ and $B=K_0(\pt)\simeq \Z$ with $\phi\colon K_0^T(\pt)\to K_0(\pt)$ being the  forgetful morphism:
\[
\phi\colon\sum k_i\lambda_i\mapsto \sum k_i.
\]

\section{Preliminaries on polytopes}\label{sec:polyprelim}

In this section we collect the geometric background needed for the
constructions in the rest of the paper.
We recall the formalism of virtual polytopes and convex chains,
introduced by Pukhlikov and Khovanskii in~\cite{PK92}, as well as
the notion of a valuation on the space of polytopes and its principal
example, the Ehrhart polynomial.
We then introduce the integer point transform and its projected
version, which serve as the equivariant analogue of the Ehrhart
polynomial in the description of equivariant $K$-theory.
The section concludes with Brion's theorem, which is the key tool
for extending identities between the projected integer point transform
and the equivariant Euler characteristic from a full-dimensional cone
in the weight lattice to the whole of~$\Lambda$
(see the proof of Proposition~\ref{prop:virtual equality}).

Throughout this section the convex-chain and valuation constructions
are carried out over~$\R$, while the integer point transform and the
shift-operator constructions of Section~\ref{sec:PKalgebras} live
over~$\Z$ (or~$\Q$ after tensoring).
The passage between the two settings is made precise in
Subsection~\ref{ssec:ehrhart}: the Ehrhart polynomial is an
integer-valued function on the lattice~$\Pm_\Sigma$ of virtual
polytopes, and it is this integrality that makes the discrete
Pukhlikov--Khovanskii machinery of Section~\ref{sec:PKalgebras}
applicable.

The reader familiar with the theory of virtual polytopes and valuations
may wish to go directly to Subsection~\ref{ssec:linfamilies}, where we
introduce the central geometric objects of the paper --- linear families
of polytopes and their $K$-rings --- which serve as polyhedral models
for the $K$-theory of toric varieties and flag varieties studied in
Sections~\ref{sec:polymodels}--\ref{sec:flags}.

A detailed exposition of virtual polytopes and convex chains can be
found in the survey~\cite{panina} by G.\,Panina and I.\,Streinu.

% ------------------------------------------------------------
\subsection{Virtual polytopes}\label{ssec:virtpoly}
% ------------------------------------------------------------

We call a polytope $P\subset \R^d$ \emph{integral} if all vertices of
$P$ belong to the integer lattice $\Z^d\subset \R^d$.
Denote by $\Pm^+$ the set of all integral polytopes in $\R^d$.
The set $\Pm^+$ has a structure of an abelian semigroup with respect to
Minkowski addition:
\[
P + Q = \{x+y\mid x\in P,\quad y\in Q\}.
\]
It is easy to check that $\Pm^+$ has the cancellation property, i.e.
\[
P_1 +Q = P_2 +Q \quad \text{if and only if} \quad P_1 = P_2.
\]
Thus $\Pm^+$ embeds into its Grothendieck group, which we denote
by~$\Pm$, i.e.\ the group of formal differences of elements of~$\Pm^+$.
The space of integral virtual polytopes~$\Pm$ is a free abelian group.

A virtual polytope $P$ is uniquely described by its
\emph{support function} $H_P\colon (\R^d)^\vee \to \R$ given by
\[
H_Q(\psi) := \min_{x\in Q}\psi(x)
\]
for a convex polytope $Q$, and extended by linearity to virtual
polytopes.
For an integral virtual polytope~$P$, the support function~$H_P$ is a
piecewise linear function that attains integer values on
$(\Z^d)^\vee$.
The cones of linearity of $H_P$ form a fan in $(\R^d)^\vee$.
For a given fan $\Sigma\subset (\R^d)^\vee$, let us denote by
$\Pm_\Sigma^+$ the set of integral convex polytopes~$P$ with $H_P$
linear on the cones of~$\Sigma$.
Similarly, by~$\Pm_\Sigma$ we denote the set of integral virtual
polytopes with support function linear on the cones of~$\Sigma$.
Clearly $\Pm_\Sigma$ is a free abelian group and $\Pm_\Sigma^+$ is its
subsemigroup.

Let $e_1,\ldots,e_r$ be the primitive ray generators of~$\Sigma$.
A virtual polytope in~$\Pm_\Sigma$ is then uniquely determined by
the values of~$H_P$ on~$e_1,\ldots,e_r$. This defines an embedding
$\Pm_\Sigma\hookrightarrow \Z^r$ as a sublattice.
Note that $\Pm_\Sigma$ is in general a proper sublattice of $\Z^r$:
it can have smaller rank and need not be saturated.

% ------------------------------------------------------------
\subsection{Convex chains and the McMullen algebra}\label{ssec:convchains}
% ------------------------------------------------------------

Consider a convex polytope $P$ in $\R^n$, not necessarily of maximal
dimension.
Its \emph{characteristic function} $\I_P\colon \R^n \to \{0,1\}$ is
defined by
\[
\I_P(x)=\begin{cases} 1,& x\in P,\\ 0, & x\notin P.\end{cases}
\]
A \emph{convex chain} is a function $\alpha\colon \R^n\to \R$ that can
be presented as a finite linear combination of characteristic functions
of polytopes with real coefficients:
\[
\alpha=\sum_{i=1}^N \alpha_i \I_{P_i},\qquad\alpha_i\in \R.
\]
Such a decomposition is not unique; a convex chain is a function on
$\R^n$, rather than a formal linear combination of polytopes.

Convex chains can be added as functions on $\R^n$.
Moreover, they form an algebra with respect to the
\emph{convolution product}~$*$: the product of two characteristic
functions of convex polytopes $P$ and $Q$ is the characteristic
function of their Minkowski sum,
\[
(\I_P * \I_Q)(x)=\I_{P+Q}(x),
\]
extended to all convex chains by linearity.
A non-trivial fact, proven in~\cite{PK92}, is that the product so
defined does not depend on the choice of presentation as a combination
of characteristic functions.
This turns the vector space of convex chains into a commutative
$\R$-algebra whose identity element is $\I_O$, the characteristic
function of the origin $O\in\R^n$.

For every convex polytope $P$, the element $\I_P$ is invertible in
this algebra. This fact is known under the names of \emph{Minkowski inversion}, or \emph{Ehrhart--Macdonald reciprocity}, see {\cite[Theorem~2]{PK92}}, {\cite[Theorem~4.1]{BeckRobins15}}.

\begin{theorem}[Minkowski inversion]
The inverse of $\I_P$ is given by
\[
(\I_P)^{* -1}(x)=(-1)^{\dim P}
        \I_{\mathrm{RInt}(\mathrm{Symm}\,P)},
\]
where $\mathrm{Symm}\,P$ is the polytope symmetric to $P$ with respect
to the origin, and $\mathrm{RInt}$ denotes the relative interior.
In terms of the M\"obius function of the face poset of~$P$:
\[
(\I_P)^{* -1}=\sum_{Q\subset P}\mu(Q)\I_Q,
\]
where $Q$ runs over the faces of~$P$.
\end{theorem}

\begin{corollary}
The algebra of convex chains contains a multiplicative subgroup
isomorphic to the group of virtual polytopes, formed by all
characteristic functions $\I_P$ where $P$ is a virtual polytope.
\end{corollary}

Moreover, the group of virtual polytopes ``almost'' coincides with the
full group of invertible elements of the convex-chain algebra:

\begin{theorem}[\cite{PK92}]
Each invertible convex chain~$\alpha$ is a real scalar multiple of the
characteristic function of a virtual polytope.
For an invertible integer convex chain~$\alpha$, either~$\alpha$
or~$-\alpha$ is the characteristic function of a virtual integral
polytope.
\end{theorem}

In what follows we will mostly work with \emph{integer convex chains},
defined as integer linear combinations of characteristic functions
$\I_P\colon \Z^n\to \Z$ (equivalently, convex chains with integer
coefficients identified whenever they agree on the lattice~$\Z^n$).
This ring contains a multiplicative subgroup of integral virtual
polytopes.

Note that in these constructions the convex chains corresponding to a
polytope~$P$ and its translate~$P+v$ are distinct.
Identifying them defines a different construction, known as the
\emph{McMullen polytope ring}.

\begin{definition}
The \emph{McMullen polytope ring} is the $\R$-vector space generated by
elements $[P]$, $P\in\Pm^+$, subject to the relations
\begin{equation}\label{eq:mm1}
[P]+[Q]=[P\cap Q]+[P\cup Q]
\end{equation}
for all $P,Q\in\Pm^+$ such that $P\cup Q$ is again a convex integral
polytope, and
\begin{equation}\label{eq:mm2}
[P]=[P+v]
\end{equation}
for every translation vector $v\in\R^n$.
Multiplication is given by Minkowski sum,
$[P]\cdot [Q]=[P+Q]$, extended by linearity.
\end{definition}

The identity element is the class of a point.
There is a natural surjective homomorphism from the convex-chain ring to
the McMullen polytope ring, sending $\I_P$ to $[P]$ for
$P\in\Pm^+$; see~\cite[Sec.~4.1--4.2]{panina}.

\begin{remark}\label{rem:mcmullen-vs-Sigma}
When we restrict attention to a fixed fan~$\Sigma$ and impose
translation-invariance, the lattice~$\Pm_\Sigma$ of virtual polytopes
with normal fan refining~$\Sigma$ embeds as a multiplicative subgroup
of the McMullen polytope ring.
It is this subgroup on which the Ehrhart polynomial and the shift
operators of Section~\ref{sec:PKalgebras} interact: the
translation-invariance of~$\Ehr$ (used crucially in the proof of
Theorem~\ref{thm:polalgrel} via the ideal~$J$) corresponds precisely
to the relation~\eqref{eq:mm2} in the McMullen ring.
\end{remark}

% ------------------------------------------------------------
\subsection{Valuations and the Ehrhart polynomial}\label{ssec:ehrhart}
% ------------------------------------------------------------

Let $A$ be an $\R$-vector space.

\begin{definition}
An $A$-valued function $\varphi$ on the set of convex polytopes is
called a \emph{valuation} if, whenever $P$, $Q$, and $P\cup Q$ are
convex polytopes, we have
\[
\varphi(P)+\varphi(Q)=\varphi(P\cap Q)+\varphi(P\cup Q).
\]
\end{definition}

We will mostly consider \emph{translation-invariant} valuations,
satisfying $\varphi(P)=\varphi(P+v)$ for every $v\in \R^n$.
Every valuation extends to the convex-chain algebra by
\[
\varphi\!\left(\sum a_i\,\I_{P_i}\right)=\sum a_i\varphi(P_i),
\]
and since virtual polytopes are convex chains of a special form, this
gives a well-defined extension of~$\varphi$ to all of~$\Pm$.

The two most standard examples of translation-invariant valuations are
the volume $\vol(P)$ and, for integral polytopes, the lattice-point
count $|P\cap\Z^n|$.

As shown by McMullen, translation-invariant valuations behave polynomially with respect to Minkowski addition:

%\todoillm{That statement is not true. take the integral of aritrary function $f$ over $P$ this is valuation but not a polynomial one. The mcmullen's theorem is about translation invariant valuations I guess. And in Pukhlikov-khovanskii they studied more general polynomial valuations.}
\begin{theorem}[{\cite[Theorem~6]{McMullen77}}]\label{thm:mcmullen}
Let $\varphi$ be a translation-invariant valuation.
\begin{enumerate}[label=\normalfont(\roman*)]
\item For any convex polytopes $P_1,\dots,P_n\subset \R^n$ and
      $\lambda_1,\dots,\lambda_n\in\R_{>0}$, the function
      \[
      F_\varphi(\lambda_1,\dots,\lambda_n)
        =\varphi(\lambda_1P_1+\dots+\lambda_nP_n)
      \]
      is a polynomial in $\lambda_1,\dots,\lambda_n$.
\item The analogous statement holds for integral polytopes and
      nonneg\-ative integers $\lambda_1,\dots,\lambda_n$.
\end{enumerate}
Here $\lambda P$ for positive real $\lambda$ denotes the image of $P$
under dilation by~$\lambda$.
\end{theorem}

For $\varphi=\vol$, the coefficient of $\lambda_1\cdots\lambda_n$ in
$F_\varphi$ is the \emph{mixed volume} of $P_1,\dots,P_n$.
Theorem~\ref{thm:mcmullen} was extended by Pukhlikov and Khovanskii
to the case of virtual polytopes:

\begin{theorem}[{\cite[Corollary~5]{PK92}}]\label{thm:PK-poly}
Let $\varphi$ be a translation-invariant valuation.
\begin{enumerate}[label=\normalfont(\roman*)]
\item For any virtual polytopes $P_1,\dots,P_n$ and
      $\lambda_1,\dots,\lambda_n\in\R$, the function
      $F_\varphi(\lambda_1,\dots,\lambda_n)
        =\varphi(\lambda_1P_1+\dots+\lambda_nP_n)$
      is a polynomial in $\lambda_1,\dots,\lambda_n$.
\item The analogous statement holds for integral virtual polytopes and
      integers $\lambda_1,\dots,\lambda_n$.
\end{enumerate}
\end{theorem}

\begin{remark}
There is a generalization of this theorem for the case of \emph{polynomial valuations}: the valuations that are not constant under translations, but rather depend polynomially upon them. See~\cite[Sec.~8]{PK92} for details.
\end{remark}

This theorem immediately implies the following proposition.

\begin{proposition}\label{prop:ehrhart-extension}
There exists a unique function $\Ehr\colon \Pm \to \Z$, called the
\emph{Ehrhart polynomial}, satisfying
\[
\Ehr(P) = |P\cap \Z^d| \quad\text{for every } P\in \Pm^+,
\]
whose restriction to every finite-dimensional affine subspace
of~$\Pm$ (equivalently, to every linear family of polytopes in the
sense of Subsection~\ref{ssec:linfamilies}) is a polynomial.
\end{proposition}

The existence and uniqueness of~$\Ehr$ follow from
Theorem~\ref{thm:PK-poly} applied to the lattice-point count
valuation: the number of lattice points of a lattice polytope extends
to a polynomial function on any finite-dimensional linear family, and
these polynomials are compatible across families.

\begin{remark}\label{rem:Ehr-integrality}
Although the convex-chain constructions of
Subsections~\ref{ssec:virtpoly}--\ref{ssec:ehrhart} are most naturally
set up over~$\R$, the Ehrhart polynomial~$\Ehr$ takes \emph{integer}
values on the lattice~$\Pm_\Sigma\subset\Z^r$.
It is this integrality that allows us to view~$\Ehr$ as a function
$\Pm_\Sigma\to\Z$ and to apply the shift-operator formalism of
Section~\ref{sec:PKalgebras} (which works over~$\Z$ or~$\Q$) to
define the $K$-ring~$K_\Sigma$ in
Subsection~\ref{ssec:linfamilies}.
\end{remark}

% ------------------------------------------------------------
\subsection{Integer point transform and projected integer point
transform}\label{ssec:TEhrhart}
% ------------------------------------------------------------

The integer point transform is the equivariant counterpart of the
Ehrhart polynomial: it records not merely the number of lattice points
in a polytope, but their precise positions (and hence their torus
weights).
It will be used to describe the $T$-equivariant $K$-theory in
Subsections~\ref{ssec:eqK} and~\ref{ssec:polyKflags}, and the
projected version~$\sigma_\pi$ defined below is the key ingredient in
the polyhedral realization of the Borel--Weil--Bott theorem via
string polytopes (Proposition~\ref{prop:virtual equality}).

\begin{definition}
Let $P\subset \R^n\cong\Lambda\otimes\R$ be a lattice polytope.
Its \emph{integer point transform} is the generating function
\[
\sigma(P)=\sum_{a\in P\cap \Lambda}\mathbf{t}^a
   \;\in\;\Z[t_1^{\pm1},\dots,t_n^{\pm1}],
\]
where $\mathbf{t}^a=t_1^{a_1}\cdots t_n^{a_n}$ for
$a=(a_1,\dots,a_n)\in\Z^n$.
\end{definition}

Now let $\pi\colon\Lambda\cong \Z^n\to\Z^r$ be a lattice homomorphism.
It induces a ring homomorphism
\[
\widehat\pi\colon\Z[t_1^{\pm 1},\dots,t_n^{\pm1}]
  \to\Z[y_1^{\pm 1},\dots,y_r^{\pm1}],
  \qquad\mathbf{t}^a\mapsto y^{\pi(a)}.
\]
For a lattice polytope~$P$ we denote the image of $\sigma(P)$
under~$\widehat\pi$ by $\sigma_\pi(P)$ and call it the
\emph{projected integer point transform} of~$P$ with respect to~$\pi$.

\begin{example}\label{ex:to-origin}
Let $\pi\colon \Z^n\to 0$ be the zero map.
Then $\sigma_\pi(P)=\Ehr(P)$ is the Ehrhart polynomial, so the
ordinary Ehrhart polynomial is the special case of the projected
integer point transform in which all torus weights are collapsed to a
single point.
\end{example}

\begin{example}\label{ex:gaussian}
Let $P=[0,m-1]\subset \R$ be a lattice segment and
$\pi\colon \Z\to\Z$ the identity map.
Then
\[
\sigma_\pi(P)=1+y+\dots+y^{m-1}=\frac{y^m-1}{y-1}=[m]_y
\]
is a Gaussian integer (as a function of $y$).
This is a prototype for the $q$-analogue of the dimension formula
that appears in the equivariant theory for flag varieties; compare
the Weyl character formula in Subsection~\ref{ssec:polyKflags}.
\end{example}

% ------------------------------------------------------------
\subsection{Brion's theorem}\label{ssec:brion}
% ------------------------------------------------------------

Brion's theorem expresses the integer point transform of a bounded
lattice polytope as a sum of contributions from its vertices.
It was first proved in~\cite{Brion88}; for an elementary account
see~\cite[Chapter~11]{BeckRobins15}.
In our paper Brion's theorem plays a crucial role in the proof of
Proposition~\ref{prop:virtual equality}: it provides the analytic
structure of $\sigma_\pi(S(\lambda))$ as a function of~$\lambda$
which, combined with the Weyl character formula, allows one to extend
the identity $\chi_T(G/B,\Lm_\lambda)=\sigma_\pi(S(\lambda))$ from a
full-dimensional cone in the weight lattice to all of~$\Lambda$.

Let $P\subset \R^n$ be a bounded lattice polytope.
For each vertex~$v$ of~$P$, let $K_v$ denote the tangent cone to~$P$
at~$v$.
The integer point transform $\sigma(K_v)$ is a Laurent series in
$t_1,\dots,t_n$; it can be written as a rational function
$\sigma(K_v)=f_v(t)/g_v(t)$ where $f_v$ and $g_v$ are Laurent
polynomials.

\begin{theorem}[{\cite{Brion88}}]\label{thm:brion}
The following equality of rational functions holds:
\[
\sigma(P)=\sum_{v\;\mathrm{vertex\;of}\;P} \sigma(K_v).
\]
\end{theorem}

% ------------------------------------------------------------
\subsection{$K$-ring of linear families of polytopes}%
\label{ssec:linfamilies}
% ------------------------------------------------------------

We now introduce the principal geometric objects of this paper.

Let $\Lambda$ be a lattice.
A \emph{linear family of polytopes} parametrized by~$\Lambda$ is a
group homomorphism $\phi\colon \Lambda \to \Pm$, i.e.\ a map
$\lambda\mapsto P_\lambda$ satisfying
\[
P_{\lambda+\mu} = P_\lambda + P_\mu \quad\text{for all }
\lambda,\mu \in \Lambda.
\]
Equivalently, a linear family is a Minkowski-linear map from the
lattice~$\Lambda$ to the group of virtual polytopes.

\begin{definition}\label{def:K ring of linear family}
Let $\phi\colon\Lambda\to \Pm$ be a linear family of polytopes.
Its \emph{$K$-ring} and \emph{equivariant $K$-ring} are the discrete
PK-rings (see Definition~\ref{thm:diffalg}) associated with the
Ehrhart polynomial and the integer point transform of the family,
respectively:
\[
K_\phi \;=\; \Sh(\Lambda)/\Ann\!\bigl(\Ehr(P_\lambda)\bigr),
\qquad
K_\phi^T \;=\; \Sh(\Lambda)/\Ann\!\bigl(\sigma(P_\lambda)\bigr).
\]
\end{definition}

Since $\Ehr$ is a polynomial on every finite-dimensional linear
family and the shift algebra~$\Sh(\Lambda)$ acts on polynomials via
difference operators, Proposition~\ref{prop:polyart} implies that
$K_\phi$ is an Artinian (hence finite-dimensional over~$\Q$) algebra
whenever $\Lambda$ is a lattice of finite rank.
For the equivariant ring~$K_\phi^T$ the same conclusion holds by
Theorem~\ref{thm:dualitywithcoef}.

\begin{remark}\label{rem:notation-KSigma}
In what follows we will use the shorthand notation $K_\Sigma$ and
$K_\Sigma^T$ for the $K$-ring and equivariant $K$-ring of the linear
family $\phi\colon\Pm_\Sigma\to\Pm$ given by the identity map on
virtual polytopes with normal fan refining~$\Sigma$.
\end{remark}

In this paper we focus on the (equivariant) $K$-rings of two families
of examples.
\begin{enumerate}[label=\normalfont(\roman*)]
\item \textbf{Toric varieties.}
      The linear family $\phi\colon\Pm_\Sigma\to\Pm$ associated to a
      smooth complete fan~$\Sigma$, parametrized by the finite-rank
      lattice of integral cone-wise linear functions on~$\Sigma$.
      The resulting rings $K_\Sigma$ and $K_\Sigma^T$ are studied in
      Sections~\ref{sec:polyKring} and~\ref{ssec:eqK}, and identified
      with the (equivariant) $K$-theory of the toric variety~$Y_\Sigma$
      in Section~\ref{ssec:Ktoric}.

\item \textbf{Flag varieties.}
      The linear family of Gelfand--Zetlin polytopes (or, more
      generally, string polytopes for a reductive group~$G$),
      parametrized by the character lattice~$\Lambda$ of a maximal
      torus of~$G$.
      The resulting rings are identified with the (equivariant)
      $K$-theory of $G/B$ in Sections~\ref{ssec:polyKflags}
      and~\ref{sec:flags}.
\end{enumerate}
In both cases the parametrizing lattice is of finite rank, so the
corresponding $K$-rings are Artinian by the remark above.

\section{Polytope $K$-ring}\label{sec:polyKring}
In this section we investigate the structure of the $K$-ring of linear family of (virtual) polytopes associated with a smooth complete fan $\Sigma$. We do not assume that $\Sigma$ is a projective fan, i.e. that there exists a convex polytope $\Delta$ with normal fan $\Sigma$.

\subsection{Generators and ideal of smooth polytope $K$-ring}\label{ssec:structureK}
Let  $\Sigma$ be a smooth complete fan with $|\Sigma(1)|=r$. In this case, the  evaluation of $H_P$ on integer ray generators of $\Sigma$ canonically identifies the lattice $\Pm_\Sigma$ with $\Z^r$. We will denote by $P_1,\ldots,P_r$ the corresponding basis of $\Pm_\Sigma$. In other words $P_i$ is a virtual polytope in $\Pm_\Sigma$ such that
\[
H_{P_i}(e_j)=\begin{cases}
			1, & \text{if $i=j$}\\
            0, & \text{otherwise},
		 \end{cases}
\]
where $e_1,\ldots e_r$ are the primitive ray generators of $\Sigma$. We further denote by $t_1,\ldots,t_r$ the shift operators with respect to $P_1,\ldots,P_r$. 

\begin{theorem}\label{thm:polalgrel}
Let $\Sigma\subset (\R^n)^\vee$ be a smooth complete fan with $\Sigma(1)=r$ and primitive ray generators $e_1,\ldots, e_r$. Then the polytope ring $K_\Sigma$ is given by
\[
K_\Sigma \simeq \Q[T_1^{\pm 1},\ldots,T_r^{\pm 1}]/(I+J),
\]
where $I$ is generated by products $(1-T^{-1}_{i_1})\cdots (1-T^{-1}_{i_t})$ such that $\rho_{i_1},\ldots, \rho_{i_t}\in \Sigma(1)$ are distinct and do not form a cone in $\Sigma$ and $J = \left\langle \prod_{i=1}^r T_i^{\langle u,e_i\rangle}-1 \colon u\in \Z^n \right\rangle$.
\end{theorem}

The difference operators appearing in the definition of ideal $I$ will play an important role in what follows. We will denote them by $D_i=1-T^{-1}_{i}$. More generally, for any cone $\sigma \in \Sigma$ we denote by $D_\sigma$ the corresponding product of difference operators
\[
D_\sigma = \prod_{\rho_i\in \Sigma}D_i.
\]
We will refer to  $D_\sigma$ as to \emph{face monomials}; these are the classes of faces corresponding to cones $\sigma\in \Sigma$.

The proof of Theorem~\ref{thm:polalgrel} is done in two steps. The first part is to show that the relations $I$ and $J$ are satisfied in the algebra $K_\Sigma$. Indeed, the relations $J$ are satisfied since  
\[
\prod_{i=1}^r T_i^{\langle u,e_i\rangle} \cdot \Ehr(P)= \Ehr\left(P+\sum_{i=1}^r\langle u,e_i\rangle P_i\right) = \Ehr(P + u) = \Ehr(P) \text{ for any } u\in \Z^n.
\]
The relations from $I$ are subject of the following lemma.
\begin{lemma}\label{lem:Dsig}
Let $\Sigma$ be a smooth fan and $P\in \Pm_\Sigma$ an integer (virtual) polytope, then 
\[
D_{i_1}\cdots D_{i_t}\cdot \Ehr(P) =  \Ehr(F_\sigma)
\]
where $\sigma$ is a cone spanned by   $\rho_{i_1},\ldots, \rho_{i_t}$ and $F_\sigma$ is a corresponding (virtual) face of $P$. In particular, if  $\rho_{i_1},\ldots, \rho_{i_t}$ do not form a cone, $D_{i_1}\cdots D_{i_t}\cdot \Ehr(P)=0$.
\end{lemma}
\begin{proof}
    We will use the language of convex chains, see Subsection~\ref{ssec:convchains} above. For a (virtual) polytope $P$, let $\I_P$ be the corresponding convex chain. The proof of the lemma follows the standard argument about the action of difference operators on the convex chains. See for example \cite[Proposition~6.2]{hof2020} for analogous statement.
\end{proof}

Since both sets of relations $I$ and $J$ are satisfied in $K_\Sigma$, there exists a surjection 
\[
\psi\colon \Q[T_1^{\pm 1},\ldots,T_r^{\pm 1}]/(I+J) \to K_\Sigma, \quad \psi\colon x_i\mapsto P_i.
\]
To show that $\psi$ is an isomorphism, we will then use a Morse theoretic argument which extend  results of \cite{khovanskii1986hyperplane,timorin1999analogue}.

First, let us notice that $\Q[D_1,\ldots,D_r] \subset \Q[T_1^{\pm 1},\ldots,T_r^{\pm 1}]$. In the following lemma we show that every element of $\Q[T_1^{\pm 1},\ldots,T_r^{\pm 1}]/(I+J)$ and hence of $K_\Sigma$ can be represented by a polynomial in difference operators.

\begin{lemma}
The natural projection
 \[
\Q[D_1,\ldots,D_r]/(I+J\cap\Q[D_1,\ldots,D_r]) \to \Q[T_1^{\pm 1},\ldots,T_r^{\pm 1}]/(I+J)
\]   
is surjective.
\end{lemma}

\begin{proof} Let $D_i^-=T_i-1$.  Monomials in classes $D_1,D_1^-,\dots, D_r,D_r^-$ linearly generate the ring $\Q[T_1^{\pm 1},\dots, T_r^{\pm r}]$. But $D_i^-$ can be expressed via $D_i$ modulo the ideal $I+J$. Indeed, since $T_i^{-1}=1-D_i$, we have
\[
D_i^-=T_i-1=\frac{1}{1-D_i}-1=D_i+D_i^2+D_i^3+\dots.
\]
But, according to Lemma~\ref{lem:squarefree}, all $D_i^k$ belong to $I+J$ for $k>n$, so $D_i^-$ is equal to an element of $\Q[D_1,\dots,D_r]$ modulo $I+J$. This completes the proof.
\end{proof}

%We will show that every class in $\Q[T_1^{\pm 1},\ldots,T_r^{\pm 1}]/(I+J)$ can be represented by a linear combination of the face monomials $D_\sigma$. 

\begin{lemma}\label{lem:squarefree}
    Every class of the form $\prod D_i^{k_i}$ is equivalent to a linear combination of face monomials  $D_\sigma$ modulo the ideal $I+J$. Moreover, for every face monomial  $D_\sigma$ appearing in such a linear combination, one has $\dim \sigma \geq \sum k_i$. In particular, $\prod D_i^{k_i}=0$ if $\sum k_i > n$.
\end{lemma}
\begin{proof}
    Without loss of generality we can assume that  $k_1,\ldots,k_s>0$ and $k_{s+1}=\ldots=k_{r}=0.$ If $\rho_1,\ldots,\rho_s$ do not form a cone in $\Sigma$, then by relations in $I$ the product is equal to $0$ and there is nothing to prove.

    Now assume that $\rho_1,\ldots,\rho_s$ form a cone $\sigma$ in $\Sigma$. Let us define the multiplicity of the monomial $\prod D_i^{k_i}$ to be the sum $m=\sum(k_i-1)$. In particular, the square free monomials are exactly the ones which have multiplicity 0. If multiplicity of $D_1^{k_1}\ldots D_s^{k_s}$ is $0$, the monomial is equal to $D_\sigma$ and again there is nothing to prove.

    Let the multiplicity of the monomial $D_1^{k_1}\ldots D_s^{k_s}$ be $m>0$. We will represent the monomial as linear combination of monomials with multiplicity $m-1$ and monomials of strictly higher degree. Without loss of generality we can assume that $k_1>1$. Since $\rho_1,\ldots,\rho_s$ form a cone in a smooth fan $\Sigma$, their primitive generators $e_1,\ldots, e_s$ is a subset of a basis for the lattice $\Z^n$. Hence there exists a vector $u \in \Z^n$ such that 
    \[
\langle u, e_1\rangle = -1 \text{ and } \langle u, e_i \rangle = 0 \quad \text{ for } i =2,\ldots, s.
    \]
    Hence, by relation from ideal $J$ we get
    \[
   T_1^{-1}\prod_{i=s+1}^r T_i^{\langle u, e_i \rangle} = 1
    \]
    Using that $D_i = 1 - T_i^{-1}$ we obtain:
    \[
(1-D_1)\cdot \prod_{\langle u, e_i \rangle<0}(1-D_i)^{-\langle u, e_i \rangle} = \prod_{\langle u, e_i \rangle>0}(1-D_j)^{\langle u, e_j \rangle}.
    \]
After opening the brackets we express $D_1$ as a linear combination of monomials of degree 1 that do not involve $D_1,\ldots,D_s$ and monomials of higher degree that do not involve $D_2,\ldots,D_s$. Hence the monomial $D_1^{k_1}\ldots D_s^{k_s}$ can be written as a combination of monomials of strictly smaller multiplicity or strictly higher degree.
\end{proof}

\subsection{Morse theory for polytopes}\label{ssec:Morse}
There are still relations between face monomials $D_\sigma$ in what follows we will choose the additive basis of the group $K_\Sigma$.

Let $w\in N$ be a generic vector with respect to $\Sigma$. That is $w$ does not belong to the linear span of any cone $\tau\in\Sigma$. 
Now let $\sigma\in \Sigma(n)$ be any maximal cone and let $\tau$ be its facet. We will say that $\tau$ is $w$-positive facet of $\sigma$ if $w$ and $\sigma$ are contained in the same half-space with respect to linear span of $\tau$. If $w$ and $\sigma$ contained in the opposite half spaces, we say that $\tau$ is $w$-negative facet of $\sigma$.

We will define the $w$-index of a maximal cone $\sigma$ denoted by $\ind_w(\sigma)$ to be the number of its $w$-positive facets. Note that the  only cone of $w$-index $n$ is the one containing $w$ in its relative interior. Similarly the only cone of $w$-index 0 is the one containing $-w$ in its interior.

For every cone $\tau\in \Sigma$ there exists a unique maximal cone $\sigma \in \Sigma(n)$ such that for a point $x\in \tau$, we have $x+\varepsilon w \in \sigma$ for any $0< \varepsilon$ small enough. We will call such $\sigma$ \emph{the attracting maximal cone} of $\tau$. We will say that a cone $\tau\in \Sigma$ is a \emph{separatrix cone} (with respect to $w$) if it is the intersection of all $w$-positive facets of its attracting maximal cone $\sigma_\tau$. 
The separatrix cones of $\Sigma$ are in bijection with its maximal cones. Indeed, for every $\sigma\in \Sigma(n)$ the  intersection $\tau$ of all its $w$-positive facets is a separatrix cone and  $\sigma=\sigma_\tau$ is the attracting maximal cone of $\tau$.

For a separatrix cone $\tau\in \Sigma$, we will call the corresponding difference operator $D_\tau$  \emph{a separatrix monomial}.

\begin{lemma}\label{lem:separatices_generate}
Let $\Sigma$ be a smooth complete fan. Then the ring $\Sh(\Pm_\Sigma)/(I+J)$ (and therefore $K_\Sigma$) is additively generated by separatrix monomials $D_\tau$. 
\end{lemma}

\begin{proof}
By Lemma~\ref{lem:squarefree} it is enough to show that every  face monomial is a combination of separatrix monomials.   Let us define a total order on the set of maximal cones $\Sigma(n)$ in the following way. For a pair of distinct maximal cones $\sigma_1,\sigma_2$ sharing a facet $\tau$ we will say that $\sigma_1<\sigma_2$ if and only if $\tau$ is a $w$-positive facet of $\sigma_2$ (or equivalently $w$-negative facet of $\sigma_1$). We define a total order as any completion of the transitive closure of the relation above. Moreover, we will say that $\tau_1\leq \tau_2$ if and only if $\sigma_{\tau_1}\leq \sigma_{\tau_2}$.

We  will prove that if $\tau$ is not a separatrix cone, then the corresponding operator $D_\tau$ can be expressed via monomials 
of higher degree and monomials of the same degree corresponding to smaller faces using relations in ideals $I$ and $J$.

Let $\sigma_\tau$ be the attracting maximal cone of $\tau$ as before and let $\rho_1,\ldots \rho_n$ be rays of $\sigma_\tau$. Denote further by $\tau^+$ the separatrix cone corresponding to $\sigma_\tau$, we have $\tau^+$ is a face of $\tau$, so there exists a ray of $\sigma_\tau$ contained in $\tau$ but not in $\tau^+$. We can assume that this ray is $\rho_1$, so we get
\[
D_{\tau} = D_{\tau^+}\cdot D_1\cdot D
\]
for some difference monomial $D$.

Consider a character $\lambda$ such that $\lambda(e_1)=1$ and $\lambda(e_i)=0$ for $i=2,\ldots, n$. Then we have the corresponding relation in $J$:
\[
1-D_1= \prod_{i>n} (1-D_i)^{-\lambda(e_i)},
\]
and thus
\[
D_{\tau} = \sum_{i>n} \lambda(e_i) D_i + \text{ higher order terms.}
\]
and
\[
D_1 = \sum_{i>n} -\lambda(e_i) D_{\tau^+}\cdot D_i \cdot D + \text{ higher order terms.}
\]
modulo the ideal $J$. Since the sum on the right hand side is taken over $i>n$, the difference monomials appearing in the sum are square free. Moreover, they are strictly smaller then $D_\tau$ as the corresponding cones contain $\tau^+$ but are not the face of $\sigma_{\tau^+}=\sigma_\tau$.
\end{proof}
\begin{lemma}\label{lem:separatices_independens}
    Separatrix monomials $D_\tau$ are linearly independent in $K_\Sigma$.
\end{lemma}
\begin{proof}
   Consider a linear combination of the separatrix monomials $a_1D_{\tau_1}+\ldots +a_kD_{\tau_k}$. Without loss of generality we can assume that $D_{\tau_1}$ is the maximal monomial in the linear combination.
   Let $\tau_1^-$ be the the separatrix cone corresponding to $\sigma_{\tau_1}$ with respect to $-w$. In over words, $\tau_1^-$ is the intersection of all $w$-negative facets of $\sigma_{\tau_1}$. 

   It is easy to see that $D_{\tau_1^-} \cdot D_{\tau_i}= 0$ in $K_\Sigma$ for $i=2,\ldots,k$. No the other hand $D_{\tau_1^-}\cdot D_{\tau_1}= 1$ since it is square free and $\tau_1$ and $\tau_1^-$ generate maximal cone $\sigma_\tau$. So we get
   \[
  D_{\tau_1^-}\cdot (a_1D_{\tau_1}+\ldots +a_kD_{\tau_k}) = a_1 \ne 0.
   \]
Hence $a_1D_{\tau_1}+\ldots +a_kD_{\tau_k} \ne 0$ which finishes the proof.
\end{proof}

\noindent
Now we are ready to finish the proof of Theorem~\ref{thm:polalgrel}.
\begin{proof}[Proof of Theorem~\ref{thm:polalgrel}]
    By translation invariance of $\Ehr$ and Lemma~\ref{lem:Dsig} we know that there is a surjection $\Sh(\Pm_\Sigma)/(I+J)\to K_\Sigma$. Moreover, by Lemma~\ref{lem:separatices_generate} we know that separatrix monomials additevly generate  $\Sh(\Pm_\Sigma)/(I+J)$ and by Lemma~\ref{lem:separatices_independens} that their images are linearly independent in $K_\Sigma$. Hence the above surjection $\Sh(\Pm_\Sigma)/(I+J)\to K_\Sigma$ is an isomorphism.
\end{proof}

In what follows we will also have to work with  non-smooth fans $\Sigma$ and their $K$-rings $K_\Sigma$. It is convenient in this case to reduce to the smooth case by considering any smooth subdivison $\Sigma'$ of $\Sigma$ and using Proposition~\ref{prop:extended}. Indeed, $\Pm_\Sigma \subset \Pm_{\Sigma'}$ and the Ehrhart polynomial on $\Pm_\Sigma$ is the restriction of Ehrhart polynomial on $\Pm_{\Sigma'}$. Thus there is an abelian group $M_{\Sigma',\Sigma}$ with an epimorphism $\pi\colon K_{\Sigma'}\to M_{\Sigma',\Sigma}$ and a monomorphism $\iota\colon K_\Sigma\to M_{\Sigma',\Sigma}$ such  that
   $\pi(\widetilde \alpha \widetilde \beta) =\iota(\alpha \beta)$  whenever $\pi(\widetilde \alpha)=\iota(\alpha)$ and $\pi(\widetilde \beta)=\iota(\beta)$. This allows to perform the computations in $K_{\Sigma'}$ instead of $K_\Sigma$ which we understand better.

\section{Polyhedral models for $K$-theory}\label{sec:polymodels}
In this section we will apply the results of the previous sections to computation of $K$-theory of toric varieties as well as generalized flag varieties. In what follows we assume basic knowledge of toric geometry and geometry of flag varieties. We refer to~\cite{coxtoric} for further details on the former and to \cite{Brion05} for the latter.

\subsection{$K$-ring of a toric variety}\label{ssec:Ktoric}

For a smooth algebraic variety $X$ we denote by $K_0(X)$ free abelian group generated by isomorphism classes of coherent sheaves on $X$ up to the relation $[\Vm]+[\Um]=[\Wm]$ whenever there is a short exact sequence $0\to \Vm\to \Wm\to\Um\to0$. The subgroup generated by classes of vector bundles is denoted by $K^0(X)$. For a smooth variety $X$ the inclusion $K^0(X)\hookrightarrow K_0(X)$ is an isomorphism. In this case, we define the ring structure on $K_0(X)$ via $[\Vm]\cdot[\Um]=[\Vm\otimes \Um]$. In what follows we will work with rational $K$-theory $K_0(X)\otimes_\Z \Q$ (as $K$-theory of toric and flag varieties is torsion free the rational $K$-theory carries the same information). To simplify the notation we will denote the rational $K$-theory of $X$ by $K_0(X)$. Finally $K$-theory admits a proper push-forward, in particular for a trivial map $f\colon X\to \mathrm{pt}$, the pushforward
$f_*\colon K_0(X) \to K_0(\mathrm{pt}) \simeq \Q$
is a linear function on $K_0(X)$ which is equal to the holomorphic Euler characteristic on the classes of sheaves:
\[
f_*([\Fm]) = \chi(X,\Fm).
\]
For a more detailed introduction to $K$-theory we refer to \cite{manin1969lectures}.

Let $T\simeq (\C^*)^n$ be an algebraic torus, $M \simeq \Z^n$ its character lattice and  $N = \Hom_\Z(M, \Z)$ its dual lattice. We denote by $M_\R= M\otimes_\Z \R$, $N_\R= N\otimes_\Z \R$ vector spaces spanned by $M$ and $N$ respectively. Let further $\Sigma$ be a smooth complete fan and $Y_\Sigma$ the corresponding toric variety. We, denote by $\Sigma(1) =\{\rho_1,\ldots,\rho_r\}$ the set of rays of $\Sigma$ and by $D_1,\ldots, D_r$ the corresponding $T$-invariant divisors. Finally, for a cone $\sigma\in \Sigma$ of dimension greater than or equal to 1 we will denote by $X_\sigma$ the closure of the $T$-orbit corresponding to $\sigma$.

The main input from toric geometry for us comes from the computation of holomorphic Euler characteristic of line bundles on $Y_\Sigma$ in terms of combinatorics of polytopes. More concretely, every line bundle $\Lm$ on $Y_\Sigma$ can be linearized, and this is equivalent (as a sheaf) to $\Om(\sum_{i=1}^r  h_i D_i)$ for some $h_1,\ldots,h_r\in \Z$. Therefore, there is a surjection 
\[
\Pm_\Sigma \to \Pic(Y_\Sigma), \quad \Delta_{h}\mapsto \Om\left(\sum_{i=1}^r  h_i D_i\right).
\]
We will denote the line bundle corresponding to a polytope $\Delta\in\Pm_\Sigma$ by $\Lm_\Delta$. The following proposition is classical \cite{Toremb}, see \cite[Proposition 3.3]{chong2026note} for formulation in terms of virtual polytopes.
\begin{proposition}\label{prop:coh}
Let $\Delta\in \Pm_\Sigma$ be an integer virtual polytope and $\Lm_\Delta$ the corresponding line bundle on $Y_\Sigma$.  Then $\chi(Y_\Sigma, \Lm_\Delta) = \Ehr(\Delta)$.
\end{proposition}

\begin{theorem}\label{thm:toricK}
Let $\Sigma$ be a smooth complete fan and let $Y_\Sigma$ be the corresponding toric variety. Then we have an isomorphism
\[
K_0(Y_\Sigma) \simeq K_\Sigma=\Sh(\Pm_\Sigma)/\Ann (\Ehr).
\]
\end{theorem}
\begin{proof}
First, since $Y_\Sigma$ has an algebraic cell decomposition, the Euler characteristic provides the Frobenius duality on $K_0(Y_\Sigma)$, i.e. the Euler pairing
\[
\langle \Fm,\Gm\rangle_{Eu} := \chi(Y_\Sigma,\Fm\otimes \Gm), \quad \Fm,\Gm\in K_0(Y_\Sigma)
\]
is non-degenerate. Moreover, $K_0(Y_\Sigma)$ is generated by Picard lattice $\Pic(Y_\Sigma)$. Therefore, by Theorem~\ref{thm:goralgrec} we get $K_0(Y_\Sigma)\simeq \Sh( \Pic(Y_\Sigma))/\Ann(\chi)$.
Finally the theorem follows from the correspondence of $\Pm_\Sigma$ with Picard lattice $\Pic(Y_\Sigma)$ and Proposition~\ref{prop:coh}.
\end{proof}

As a corollary of Theorem~\ref{thm:toricK} and Theorem~\ref{thm:polalgrel} we obtain the following statement.
\begin{corollary}\label{cor:Krelations}
Let $\Sigma$ be a smooth, complete fan and $Y_\Sigma$ the corresponding toric variety. Then
\[
K_0(Y_\Sigma)\simeq \Q[T_1^{\pm 1},\ldots,T_r^{\pm 1}]/(I+J),
\]
where $I$ is generated by monomials $D_{i_1}\cdots D_{i_t}$ such that $\rho_{i_1},\ldots, \rho_{i_t}\in \Sigma(1)$ are distinct and do not form a cone in $\Sigma$ and $J = \left\langle \prod_{i=1}^r T_i^{\langle u,e_i\rangle}-1 \colon u\in \Z^n \right\rangle$.
\end{corollary}

We will finish this section with a description of the classes of structure sheaves $\Om_{X_\sigma}$ of orbit closures in the polytope $K$-ring $K_\Sigma$.
\begin{proposition}
Let $\Sigma$ be a smooth fan and $Y_\Sigma$ the corresponding toric variety. Let further $\sigma\in \Sigma$ be a cone and $X_\sigma \subset Y_\Sigma$ the corresponding orbit closure. Then the class of  $\Om_{X_\sigma}$ is represented in $K_\Sigma$ by the operator $D_\sigma = \prod_{\rho_i\in \sigma}D_i$.
\end{proposition}
\begin{proof}
Indeed, since $K_0(Y_\Sigma)$ is generated by $\Pic(Y_\Sigma)$, it is enough to check that $\chi(Y_\Sigma,\Om_{X_\sigma}\otimes \Lm_\Delta)=D_\sigma\cdot \Ehr(\Delta)$ for any $\Delta\in \Pm_\Sigma$. Hence the proposition follows from Lemma~\ref{lem:Dsig} and the fact that $\chi(Y_\Sigma,\Om_{X_\sigma}\otimes \Lm_\Delta)=\chi(X_\sigma, \Lm_\Delta|_{X_\sigma}) = \Ehr(F_\sigma)$, where $F_\sigma$ is a face of $\Delta$ corresponding to $\sigma$.
\end{proof}

\subsection{Equivariant $K$-theory}\label{ssec:eqK}
In this subsection we give a polyhedral description for the $T$-equivariant $K$-theory. For this we will use the formalism developed in Subsection~\ref{ssec:dualitywithcoefficients}. Relations are described in \cite{VezzosiVistoli03}.

First, recall that for a variety with a $T$-action the equivariant $K$-theory $K_0^T$ is the quotient of free abelian group generated by isomorphism classes of $T$-\emph{equivariant} coherent sheaves modulo short exact sequence relation. 
The subgroup generated by classes of equivariant vector bundles is denoted by $K^0_T(X)$. For a smooth variety $X$ the inclusion $K^0_T(X)\hookrightarrow K_0^T(X)$ is an isomorphism. In this case, we define the ring structure on $K_0^T(X)$ via $[\Vm]\cdot[\Um]=[\Vm\otimes \Um]$. 

Finally equivariant $K$-theory admits a proper pushforward. To define the pushforward with respect to a trivial map $f\colon X\to \mathrm{pt}$, recall that $K_0^T(\mathrm{pt}) \simeq \Z[\Lambda]$ is the representation ring of the torus $T$. Then the pushforward map
$f_*\colon K_0^T(X) \to K_0^T(\mathrm{pt})$ is given by the equivariant Euler characteristic
\[
f_*([\Fm]) = \chi^T(X,\Fm) = \sum (-1)^i H^i(X,\Fm). 
\]
Here we view $H^i(X,\Fm)$ as a representation of a torus, thus the sum in the right-hand side belongs to $\Z[M]$.
For a more detailed introduction to equivariant $K$-theory in algebraic setting we refer for example to \cite[Chapter 5]{chriss1997representation}.

Now let $Y_\Sigma$ be a smooth complete toric variety associated to a fan $\Sigma\subset N$. Similar to the classical $K$-theory, the equivariant $K$-theory is generated by the classes of equivariant line bundles. Each equivariant line bundle bundle on $Y_\Sigma$ is given as $\Om(\sum_{i=1}^rh_i D_i)$ and thus is defined by a (virtual) polytope. In fact one has an isomorphism of abelian groups
\[
\Pm_\Sigma \to \Pic^T(Y_\Sigma), \quad P_{h}\mapsto \Om\left(\sum_{i=1}^r  h_i D_i\right).
\]
With a slight abuse of the notation, we will denote by $\Lm_P$ the equivariant line bundle corresponding to a (virtual) polytope $P\in \Pm_\Sigma$. Moreover, the equivariant Euler pairing is a non-degenerate pairing on the equivariant $K$-theory of smooth complete toric variety. Finally, we have the following proposition.

\begin{proposition}\label{prop:eqcoh}
Let $P\in \Pm_\Sigma$ be an integer virtual polytope and $\Lm_P$ the corresponding line bundle on $Y_\Sigma$. Then the equivariant Euler characteristic of $\Lm_P$ is given by:
\[
\chi^T(Y_\Sigma, \Lm_P) = \sum_{u\in P\cap M}e^u.
\]
\end{proposition}

Recall that we denote by $\sigma\colon \Pm_\Sigma \to \Z[M]$ the function $P\mapsto \sum_{u\in P\cap M}e^u$. As an immediate corollary of Theorem~\ref{thm:dualitywithcoef} we get the following description of the equivariant $K$-theory of a smooth complete toric variety.

\begin{theorem}\label{thm:toricKeq}
Let $\Sigma$ be a smooth, complete fan and let $Y_\Sigma$ be the corresponding toric variety. Then we have an isomorphism $K_0^T(Y_\Sigma) \simeq \Sh(\Pm_\Sigma)/\Ann (\sigma)$.
\end{theorem}

As a corollary, we obtain a classical generators and relations presentation of equivariant $K$-theory of smooth complete toric variety.
\begin{corollary}
    Let $\Sigma$ be a smooth complete fan and let $Y_\Sigma$ be the corresponding toric variety. Then we have an isomorphism
\[
K_0^T(Y_\Sigma)\simeq \Q[T_1^{\pm 1},\ldots,T_r^{\pm 1}]/I,
\]
where $I$ is generated by monomials $D_{i_1}\cdots D_{i_t}$ such that $\rho_{i_1},\ldots, \rho_{i_t}\in \Sigma(1)$ are distinct and do not form a cone in $\Sigma$.
\end{corollary}
\begin{proof}
   Note that $\Ehr(P)= \phi\circ \sigma(P)$, where $\phi\colon\Z[M]\to \Z$ is the evaluation of a Laurent polynomial at $(1,\ldots,1)$. By Proposition~\ref{prop:basechange}, we get that $\Ann(\sigma)\subset \Ann(\Ehr)$ which is given by the sum of two ideals $I+J$ by Corollary~\ref{cor:Krelations}. A direct computation then shows that $\Ann(\sigma)=I$. 
\end{proof}

\subsection{Polyhedral model for $K$-theory of $G/B$}\label{ssec:polyKflags}

In this section we apply the theory of of discrete duality algebras to give a polyhedral description of ($T$-equivariant) $K$-theory of generalized flag varieties $G/B$ for a reductive group $G$ and a Borel subgroup $B\subset G$.  In further sections we will investigate the case $G= \GL(n)$ in more detail.

Let $G$ be a reductive algebraic group of rank $n$ and let $B\subset G$ be a Borel subgroup of $G$. In what follows, we will denote by $N=\dim(G/B)$ the dimension of generalized full flag variety associated to $G$. Let $T$ be the maximal torus corresponding to $B$, and $W=N(T)/T$ the Weyl group of $G$. Denote by $\Lambda=X^*(T)$ the weight lattice of $G$ (i.e., the character lattice of $T$), and let $\lambda\in\Lambda$ be a character. It corresponds to the line bundle 
\[
\calL_\lambda=G\times^{B}\C_{-\lambda}\to G/B,
\]
where $B$ acts on $\C_{-\lambda}$ via $-\lambda$, and the unipotent radical of $B$ acts trivially. The $T$-equivariant Euler characteristic $\chi_T(G/B,\calL_\lambda)$ is then an element of the representation ring $R(T)$ of $T$:
\[
\chi_T(G/B,\calL_\lambda)=\sum_{i\geq 0}(-1)^i[H^i(G/B,\calL_\lambda)]\in R(T)=\Z[\Lambda].
\]

The classical Borel--Weil--Bott theorem (\cite{Bott57}) states that
\begin{equation}\label{eq:BWB}
\chi_T(G/B,\calL_\lambda)=(-1)^w[V(w\cdot \lambda)],
\end{equation}
if there exists a unique  $w\in W$ such that $w\cdot\lambda=w(\lambda+\rho)-\rho$ is dominant, and zero otherwise. So the $T$-equivariant Euler characteristic of $\calL_\lambda$ can be computed using the \emph{Weyl character formula}:
\begin{equation}\label{eq:Weyl character}
\chi_T(G/B,\calL_\lambda)=\frac{\sum_{u\in W}(-1)^{\ell(u)}e^{u(\lambda+\rho)}}{\prod_{\alpha\in\Delta^+}(e^{\alpha/2}-e^{-\alpha/2})},
\end{equation}
where the product is taken over the system of positive roots $\Delta_+$,  and $\rho=\frac{1}{2}\sum_{\alpha\in\Delta^+}\alpha$ stands for the half-sum of positive roots.

We want to view this as a map
\[
F_G^T \colon \Lambda\to \Z[\Lambda], \qquad \lambda \mapsto \frac{\sum_{u\in W}(-1)^{\ell(u)}e^{u(\lambda+\rho)}}{\prod_{\alpha\in\Delta^+}(e^{\alpha/2}-e^{-\alpha/2})}.
\]
%\todo[inline]{Should we switch $e$ to $\bf t$ in the Weyl formula to get an element of $\Z[\Lambda]$?}

Forgetting the $T$-action (i.e. setting all $e^\alpha\to 1$), we recover the usual (non-equivariant) Euler characteristic of~$\calL_\lambda$ via the \emph{Weyl dimension formula}:
\[
\chi(G/B,\calL_\lambda)=(-1)^w\dim V(w\cdot \lambda)=\prod_{\alpha\in \Delta^+}\frac{( w(\lambda+\rho),\alpha) }{(\rho,\alpha)}.
\]

Thus, $\chi(G/B,\calL_\lambda)$ is a polynomial in $\lambda_1,\dots,\lambda_n$, sometimes referred to as the \emph{Weyl polynomial} and denoted by $F_G(\lambda)$. This is an inhomogeneous polynomial of degree $|\Delta^+|=\dim(G/B)$.

\begin{example}
If $G=\GL(n)$, the Weyl polynomial equals $F_{\GL(n)}(\lambda)=\displaystyle\prod_{i<j}\frac{\lambda_i-\lambda_j-i+j}{j-i}$.
\end{example}

The previous discussion immediately implies the following theorem about the $K$-group of the full flag variety $G/B$.

\begin{theorem}\label{thm:kflag-general}     Let $G/B$ be a generalized full flag variety, $X=X(T)$ the character lattice of maximal torus, and $F_G$ the Weyl polynomial of $G$. Then we have the following identification:
    \[
    K_0(G/B) = \Sh(\Lambda)/\Ann(F_G),
    \]
    \[
    K_0^T(G/B) = \Sh(\Lambda)/\Ann(F_G^T).
    \]
\end{theorem}

\begin{proof}
The proof is analogous to the proof of Theorem~\ref{thm:toricK}. Indeed, $G/B$ has an algebraic cell decomposition, and $K_0(G/B)$ is generated by $\Pic(G/B)$, that can be identified with $\Lambda$. Therefore, the statement follows from Theorems~\ref{thm:goralgrec} and~\ref{thm:dualitywithcoef} by applying Borel--Weil--Bott theorem.
\end{proof}

The Weyl polynomial has the following polyhedral interpretation. For a dominant weight $\lambda$, there exists a \emph{string polytope} $S(\lambda) \subseteq \R^N$ such that
\[
    |S(\lambda)\cap \Lambda| = \dim V(\lambda)\text{.}
\]
Moreover, there exists a linear projection $\pi\colon \R^N\to \Lambda\otimes \R$, such that for every dominant $\lambda$ the class of $V(\lambda)$ in the representation ring of $T$ can be computed as a projected integral point transform of $S(\lambda)$:
\begin{equation}\label{eq:module point transform}
[V(\lambda)] = \sum_{a\in S(\lambda)\cap \Z^N} {\bf t}^{\pi(a)} = \sigma_\pi(S(\lambda)).
\end{equation}

This equality follows from Littelmann's results on string polytopes and crystal bases: indeed, as stated in \cite[Proposition 1.5]{littelmanncones} the set of integral points in the string polytope $S(\lambda)\subset \R^N$ parametrizes the elements of the crystal basis $\Bm(\lambda)$, and the projected integer point transform corresponding to the map $\pi\colon \R^N\to\Lambda\otimes\R$ takes $S(\lambda)\subset \R^N$ to the character of $V(\lambda)$. For a reformulation of this statement in terms of Newton--Okounkov bodies, see~\cite{Kaveh15}.

The construction of string polytopes depends on the choice of a reduced expression $\underline w_0$ of the longest element $w_0$ in the Weyl group, so we fix some reduced expression $\underline w_0$ (see \cite{littelmanncones, BerZel}). In type $A$ string polytopes in particular recover classical Gelfand--Zetlin polytopes.
% There exists a convex cone $\Cm_{\underline w_0}\subseteq \X_\R\oplus \R^N$ called \emph{weighted string cone} such that 
% \[
%   S_\lambda = \pi^{-1}(\lambda)\cap \Cm_{\underline w_0} \text{,}
% \]
% where $\pi \colon \X_\R\oplus \R^N\to \X_\R$ is the natural projection 

As in the case of Gelfand--Zetlin polytopes, string polytopes are homogeneous in $\lambda$, i.e. $S(k\lambda)=kS(\lambda)$.
However, string polytopes are not necessarily additive, i.e. $S(\lambda+\mu)=S(\lambda) + S(\mu)$ is not true in general. Nevertheless, there exists a fan structure $\Fm$ on positive Weyl chamber, such that for every cone $\tau \in \Fm$ the string polytopes form a Minkowski-linear family of polytopes (see for instance \cite[Proposition 1.4]{kavvil}). Let us fix a cone of maximal dimension $\tau\in \Fm$ For any $\lambda \in \Lambda$, we will denote by $S(\lambda)$ the virtual polytope extending linear family of string polytopes on $\tau$.

\begin{proposition}\label{prop:virtual equality}
    For any $\lambda\in \Lambda$ we have equalities
    \[
    \chi_T(G/B,\Lm_\lambda) = \sigma_\pi(S(\lambda)), \qquad   \chi(G/B,\Lm_\lambda) = \Ehr(S(\lambda)).
    \]
\end{proposition}
\begin{proof}
    First of all, it is enough to check the first equality since the second one is obtained from the former by substituting $t=1$. Further, the first equality is true for any $\lambda \in \tau$ by Equations~\eqref{eq:BWB} and~\eqref{eq:module point transform}, so we need to prove that the equality extends to the whole $\Lambda$.

    By the Weyl character formula \eqref{eq:Weyl character} the left hand side of the equality can be written as a sum of rational functions 
    \[
     \chi_T(G/B,\Lm_\lambda) = \sum_i \frac{A_i(\lambda)}{B_i},
    \]
    where $B_i\in \Z[\Lambda]$ are fixed Laurent polynomials and $A_i(\lambda)$ are Laurent polynomials of the form:
    \[
   A_i(\lambda) = \sum_{M} a_{i,M}{\bf t}^{M\lambda},
    \]
    where $M\in \mathrm{Mat}_{n\times n}(\Z)$ is an integer matrix. In other words, the coefficients of $A_i(\lambda)$ are fixed, and the degrees depend linearly on $\lambda$.

    On the other hand, by Brion's theorem (Theorem~\ref{thm:brion}) we get a similar expression for the integer point transform of $S(\lambda)$:
    \[
    \sigma(S(\lambda)) = \sum_j \frac{C_j(\lambda)}{D_j},
    \]
    where the denominators $D_j$ are fixed Laurent polynomials and the degrees of monomials in $C_j(\lambda)$ depend linearly on $\lambda$.
    Both expressions are valid for any $\lambda\in \Lambda$. Moreover, since the projected integer point transform $\sigma_\pi(S_\lambda)$ is obtained by substituting ${\bf t}^u$ by ${\bf t
    }^{\pi(u)}$, it is enough to show that for every $\lambda\in \Lambda$ the function 
    \begin{equation}\label{eq:afuny Laurent}
    \prod D_j^{\pi} \cdot \sum_{i} \left(A_i(\lambda) \cdot\prod_{k\ne i} B_k \right) - \prod B_i \cdot \sum \left( C_j(\lambda)^\pi \cdot \prod_{\ell\ne j} D_\ell^\pi \right)
    \end{equation}
    is identically~$0$. Here by $C_j(\lambda)^\pi$ and $D_j^\pi$ we denote the result of substitution induced by $\pi$ of Laurent polynomials $C_j(\lambda)$ and $D_j$.

    The function defined in \eqref{eq:afuny Laurent} is a Laurent polynomial with fixed coefficients and degrees depending affinely on $\lambda$. Moreover, we know that for a full dimensional  cone $\tau\subset \Lambda$  it is identically equal to zero. Since the cancellation of monomials can only happen along finitely many affine subspaces of $\Lambda$ and they cannot cover the whole cone $\tau$, we obtain that the function defined in \eqref{eq:afuny Laurent} is identically equal to~0 for any $\lambda \in \Lambda$.
\end{proof}

As an immediate corollary, we obtain the following theorem.

\begin{theorem}\label{thm:kflag-string}
    Let $G/B$ be a (generalized) full flag variety, $\Lambda=X(T)$ the character lattice of maximal torus and $S(\lambda)$ (possibly virtual) string polytope corresponding to $\lambda\in \Lambda$. Then we have the following identifications:
    \[
    K_0(G/B) = \Sh(\Lambda)/\Ann(\Ehr(S(\lambda))), \qquad     K_0^T(G/B) = \Sh(\Lambda)/\Ann(\sigma_\pi(S(\lambda))).
    \]  
\end{theorem}

\begin{remark}
    For $G=\GL(n)$ the statement of Proposition~\ref{prop:virtual equality} follows from the main result of \cite{Makhlin16}, where the author proves directly that Brion's theorem recovers the Weyl character formula by showing that the contributions of all non-simple vertices of Gelfand--Zetlin polytope to the projected integer point transform vanish. It would be interesting to see the generalization of results of~\cite{Makhlin16} to other types.
\end{remark}

\begin{remark}
    Note that the linear family of string polytopes~$S(\lambda)$ depends on the choice of a reduced expression $\underline w_0$ of the longest element $w_0\in W$ in the Weyl group, as well as the choice of maximal cone $\tau\in \Fm$. Thus different choices possibly provide different polyhedral models for $K_0(G/B)$ and $K_0^T(G/B)$. It would be interesting to investigate these different presentations as well as isomorphisms between them.
\end{remark}

\section{Combinatorics of Gelfand--Zetlin polytopes}\label{sec:gz}
In the rest of the paper
we will apply results of Subsection~\ref{ssec:polyKflags} to  polyhedral representatives for structure sheaves of Schubert classes in $\GL(n)/B$. First we recall the definition and some combinatorial properties of Gelfand--Zetlin polytopes.

\subsection{Gelfand--Zetlin polytopes}\label{ssec:gzpoly}

Take a strictly decreasing sequence of integers $\lambda=(\lambda_1>\lambda_2>\dots>\lambda_n)$. Consider a triangular tableau of the following form (it is called  \emph{Gelfand--Zetlin tableau}):
\begin{equation}\label{eq:GZ}
\begin{array}{ccccccccc}
 \lambda_n && \lambda_{n-1} &&\lambda_{n-2} && \dots && \lambda_1\\
 & x_{1{1}} && x_{1{2}} && \dots && x_{1,{n-1}}\\
 && x_{21}  && \dots && x_{2,{n-2}}\\
&&& \ddots & \vdots\\
&&&&x_{n-1,1}
\end{array}
\end{equation}
We will interpret $x_{ij}$, where $i+j\leq n$, as coordinates in $\R^N$, where $N=\frac{n(n-1)}2$.  This tableau can be viewed as a set of inequalities on the coordinates in the following way: for each triangle $\begin{array}{ccc} a && b\\& c\end{array}$ in this tableau, impose the inequalities $a\leq c\leq b$. This system of inequalities defines a bounded nondegenerate polytope in $\R^N$. This polytope is called a \emph{Gelfand--Zetlin polytope}; we will denote it by $GZ(\lambda)$.

Gelfand--Zetlin polytopes were introduced by I.~M.~Gelfand and M.~L.~Zetlin\footnote{Sometimes also spelled Cetlin or Tsetlin.} in~\cite{GelfandZetlin50}. The integer points in $GZ(\lambda)$ index a special basis, called the Gelfand--Zetlin basis, in the irreducible representation $V(\lambda)$ with the highest weight $\lambda$ of the group $\GL(n)$. In particular, the number of integer points in $GZ(\lambda)$ is equal to $\dim V(\lambda)$. 

 \begin{example}\label{ex:GZ3} For $n=3$, the polytope $GZ(\lambda)$ and the corresponding Gelfand--Zetlin tableau and the polytope is shown in Figure~\ref{fig:GZpolytope}.
 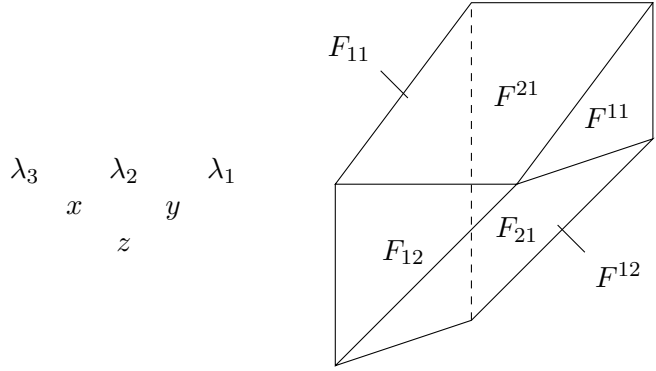
\begin{figure}[h!]
 \begin{center}
 $\begin{array}{ccccc}
   \lambda_3 && \lambda_2 &&\lambda_1\\
   & x && y\\
   && z
  \end{array}$ \qquad  %\includegraphics[scale=0.7]{gz-1}
 \begin{tikzpicture}[baseline={(0,2)},scale=0.6]
    \draw (0,0)--(3,1)--(7,5)--(4,4)--(0,4)--(3,8)--(7,8)--(4,4)--(0,0)--(0,4);
    \draw (7,5)--(7,8);
    \draw[dashed] (3,1)--(3,8);
    \node at (1.5,2.5) {$F_{12}$};
    \node at (4,3) {$F_{21}$};
    \node at (6,5.5) {$F^{11}$};
    \node at (4,6) {$F^{21}$};
    \node[below right] at (5.5,2.5) {$F^{12}$};
    \node[above left] at (1,6.5) {$F_{11}$};
    \draw (5.5,2.5)--(4.9,3.1);
    \draw (1.6,5.9)--(1,6.5);
\end{tikzpicture}
\end{center}
 \caption{Gelfand--Zetlin polytope in dimension 3}\label{fig:GZpolytope}
 \end{figure}
\end{example}

The following proposition is immediate.

\begin{proposition}\label{prop:ehrhart-gz} For a given $n$, all Gelfand--Zetlin polytopes have the same normal fan. The Ehrhart polynomial of $GZ(\lambda)$ is equal to the Weyl polynomial of $\GL(n)$:
\[
 \Ehr(GZ(\lambda))=F_{\GL(n)}=\prod_{i<j}\frac{\lambda_i-\lambda_j-i+j}{j-i}.
\]
\end{proposition}

We denote the lattice of (possibly virtual) integer Gelfand--Zetlin polytopes by $\Pm_{GZ}$. Theorem~\ref{thm:kflag-general} together with Proposition~\ref{prop:ehrhart-gz} implies the following corollary.

\begin{corollary}\label{cor:ktheory-gl}
Let $\Fl(n)=\GL(n)/B$ be the variety of complete flags in $\C^n$. Its $K$-group is isomorphic to
$K_0(\Fl(n))\cong  K_{GZ}=\Sh(\Pm_{GZ})/\Ann\Ehr(GZ(\lambda))$.
\end{corollary} 

%\todo[inline]{Picture, check numeration, more details on projection onto weight polytope}

\subsection{Projection of Gelfand--Zetlin polytope onto weight polytope}\label{ssec:GZprojection}

Consider the map $\pr$ taking each row of the Gelfand--Zetlin pattern, starting from the row of $\lambda_i$, into the sum of its entries minus the sum of the entries in the next row:
\[
\pr\colon \R^N\to\R^{n}\cong\Lambda\otimes \R,
\]
\[
(x_{11},x_{12},\dots,x_{n-1,1})\mapsto \left(\sum_{i=1}^n\lambda_i-\sum_{i=1}^{n-1} x_{1i},\sum_{i=1}^{n-1} x_{1i}-\sum_{i=1}^{n-2} x_{2i},\dots,(x_{n-2,1}+x_{n-2,2})-x_{n-1,1},x_{n-1,1}\right).
\]
This map projects  $GZ(\lambda)$ onto the weight polytope of $V(\lambda)$, considered as an $(n-1)$-dimensional polytope in the affine hyperspace in $\R^n$ defined by the equation $y_1+\dots+y_n=|\lambda|$. Here we denote the standard coordinates in~$\R^n$ by $y_1,\dots,y_n$.

\begin{example} In Example~\ref{ex:GZ3}, the image of $GZ(\lambda)$ under the projection 
\[
\pr\colon (x,y,z)\mapsto (\lambda_1+\lambda_2+\lambda_3-x-y,x+y-z, z)
\]
is a hexagon in the plane $x+y+z=\lambda_1+\lambda_2+\lambda_3$. This hexagon is the weight polytope of the representation~$V(\lambda)$ of~$\GL(3)$.
\end{example}

In the case of Gelfand--Zetlin polytope the formula (\ref{eq:module point transform}) on p.~\pageref{eq:module point transform} takes the following form. It also readily follows from the original construction of Gelfand--Zetlin basis, cf.~\cite{GelfandZetlin50}.

\begin{proposition} Let $\lambda\in\Lambda_+$. The projected integer point transform of $GZ(\lambda)$ with respect to the projection $\pr\colon\R^N\to \R^{n}$ equals the character of $V(\lambda)$:
\[
\sigma_\pr(GZ(\lambda))(y_1,\dots,y_n)= \frac{\det(y_i^{\lambda_j+n-j})_{1\leq i<j\leq n}}{\det(y_i^{n-j})_{1\leq i<j\leq n}}=\frac{\det(y_i^{\lambda_j+n-j})_{1\leq i<j\leq n}}{\prod_{1\leq i<j\leq n}(y_i-y_j)}.
\]
\end{proposition}
 
\subsection{Faces of Gelfand--Zetlin polytopes}\label{ssec:facesGZ}

 Let us describe the set of faces of the Gelfand--Zetlin polytope. The polytope is defined by a set of inequalities, represented by the Gelfand--Zetlin pattern~(\ref{eq:GZ}). Each face is obtained by turning some of these inequalities into equalities. 
 
 In particular, each facet (i.e. face of codimension 1) is defined by a unique equation:   $x_{ij}=x_{i-1,j}$ or $x_{ij}=x_{i-1,j+1}$ for some pair $(i,j)$, where $1\leq i\leq n-1$ and $i+j\leq n$. We also formally set $x_{0,k}=\lambda_{n-k+1}$. Denote the facets of the first type by $F_{ij}$, and the facets of the second type by $F^{ij}$. The total number of facets is thus equal to $2{\binom{n}{2}}=n(n-1)$.

 Every face of smaller dimension can be obtained as the intersection of certain facets. We will denote it by the indices corresponding to each of them:
 \[
 F_{i_1j_1,\dots,i_kj_k}^{r_1s_1,\dots,r_ms_m}=(F_{i_1j_1}\cap \dots \cap F_{i_kj_k})\cap (F^{r_1s_1}\cap\dots\cap F^{r_ms_m}).
 \]
Note that since $GZ(\lambda)$ is not simple, a presentation as the intersection of facets is not unique. 

\begin{definition} We say that $F$ is a \emph{Kogan face} (respectively \emph{dual Kogan face}) if it is obtained as the intersection of facets only of the form $F_{ij}$ (respectively $F^{ij}$).
\end{definition}

The minimal (by inclusion) Kogan face, contained in all Kogan faces, is the vertex with the smallest possible values of coordinates. It is defined by the equations
\begin{eqnarray*}
\lambda_1&=&x_{11}=x_{21}=\dots=x_{n-2,1}=x_{n-1,1},\\   
\lambda_2&=&x_{12}=x_{22}=\dots=x_{n-2,2},\\
&\dots\\
\lambda_{n-1}&=&x_{1,n-1}.
\end{eqnarray*}
Similarly, the minimal dual Kogan face is the vertex defined by the equations
\begin{eqnarray*}
\lambda_n&=&x_{1,n-1}=x_{2,n-2}=\dots=x_{n-2,2}=x_{n-1,1},\\   
\lambda_{n-1}&=&x_{1,n-2}=x_{2,n-3}=\dots=x_{n-2,1},\\
&\dots\\
\lambda_{2}&=&x_{11}.
\end{eqnarray*}
We will refer to these vertices as the \emph{Kogan vertex} and the \emph{dual Kogan vertex}, respectively. They are simple, so the equations defining them are independent, and there are $2^{\binom{n}{2}}$ Kogan faces in total, including the polytope itself, and the same number of dual Kogan faces.

\subsection{Relations in the polytope ring of $GZ(\lambda)$}\label{ssec:relations}

Here we describe the relations in the polytope ring of $GZ(\lambda)$. These relations are obtained from the equation~(\ref{eq:mm2}) of translation invariance:
\[
[GZ(\lambda)]=[GZ(\lambda)+v].
\]
Here it is enough to take $v=e_{rs}$ for all standard basis vectors $e_{rs}$ of $\R^{\binom{n}{2}}$.

Consider the Gelfand--Zetlin polytope $GZ(\lambda)$ defined by a dominant highest weight $\lambda=(\lambda_1>\dots>\lambda_n)$. As before, set $x_{0j}=\lambda_{n+1-j}$. Then $GZ(\lambda)$ is defined by $2^{\binom{n}{2}+1}$  inequalities
\[
x_{i-1,j}\leq x_{ij}\quad \text{and}\quad x_{ij}\leq x_{i-1,j+1}.
\]
Now shift this polytope by a standard basis vector $e_{rs}$. This means that for the shifted polytope $\widehat{GZ(\lambda)}=GZ(\lambda)+e_{ij}$, all inequalities not involving $x_{rs}$ remain the same, while the four inequalities with $x_{rs}$ will look as follows:
\begin{eqnarray*}
x_{r-1,s}&\leq& x_{rs}-1,\\ 
x_{rs}-1&\leq& x_{r-1,s+1},\\
x_{r+1,s-1}&\leq& x_{rs}-1,\\
x_{rs}-1&\leq& x_{r+1,s}.
\end{eqnarray*}
So the difference $\widehat{GZ(\lambda)}\setminus GZ(\lambda)$ (as a set of lattice points) is the union of two facets of $\widehat{GZ(\lambda)}$, namely, those obtained from $F^{rs}$ and $F_{r+1,s}$ by shifting by $e_{rs}$, while the difference $GZ(\lambda)\setminus \widehat{GZ(\lambda)}$ is the union of two faces $F_{rs}$ and $F^{r+1,s-1}$. Thus, in terms of integral convex chains we have the following identity: 
\[
\I_{\widehat{GZ(\lambda)}}-\I_{GZ(\lambda)}=(\I_{F^{rs}}+\I_{F_{r+1,s}}-\I_{F_{r+1,s}^{rs}})-(\I_{F_{rs}}+\I_{F^{r+1,s-1}}-\I_{F_{rs}^{r+1,s-1}}).
\]
This proves the following theorem.

\begin{theorem}[Six-term relations]\label{thm:6term} In the polytope ring of a Gelfand--Zetlin polytope, the following relations hold for every $(r,s)$, where $1\leq r\leq n-1$ and $1\leq s\leq n-r$:
\[
[F^{rs}]+[F_{r+1,s}]-[F_{r+1,s}^{rs}]=[F_{rs}]+[F^{r+1,s-1}]-[F_{rs}^{r+1,s-1}].
\]
Here the summands corresponding to the pairs of indices that are out of range are set to be zero.
\end{theorem}

The six-term relations can be graphically represented as shown in~Figure~\ref{fig:6term} below, by means of the diagrams representing faces of the Gelfand--Zetlin polytope.

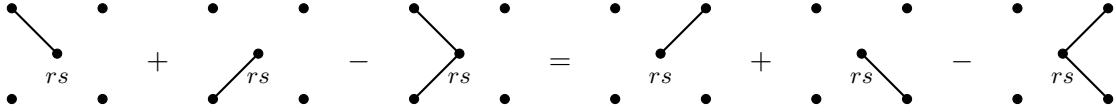
\begin{figure}[h!]
\begin{center}
 \begin{tikzpicture}[baseline={(0,1)},scale=0.6]
 \draw[fill] (-1,3) circle [radius=0.1];
 \draw[fill] (1,3) circle [radius=0.1];
 \draw[fill] (0,2) circle [radius=0.1];
 \draw[fill] (-1,1) circle [radius=0.1];
 \draw[fill] (1,1) circle [radius=0.1];
 \draw[thick] (0,2)--(-1,3);
 \node[below] at (0,1.8) {\footnotesize{$rs$}};
\end{tikzpicture} $\quad +\quad$
\begin{tikzpicture}[baseline={(0,1)},scale=0.6]
 \draw[fill] (-1,3) circle [radius=0.1];
 \draw[fill] (1,3) circle [radius=0.1];
 \draw[fill] (0,2) circle [radius=0.1];
 \draw[fill] (-1,1) circle [radius=0.1];
 \draw[fill] (1,1) circle [radius=0.1];
 \draw[thick] (0,2)--(-1,1);
 \node[below] at (0,1.8) {\footnotesize{$rs$}};
\end{tikzpicture} $\quad -\quad$ 
\begin{tikzpicture}[baseline={(0,1)},scale=0.6]
 \draw[fill] (-1,3) circle [radius=0.1];
 \draw[fill] (1,3) circle [radius=0.1];
 \draw[fill] (0,2) circle [radius=0.1];
 \draw[fill] (-1,1) circle [radius=0.1];
 \draw[fill] (1,1) circle [radius=0.1];
 \draw[thick] (0,2)--(-1,3);
 \draw[thick] (0,2)--(-1,1);
 \node[below] at (0,1.8) {\footnotesize{$rs$}};
\end{tikzpicture} $\quad =\quad$
 \begin{tikzpicture}[baseline={(0,1)},scale=0.6]
 \draw[fill] (-1,3) circle [radius=0.1];
 \draw[fill] (1,3) circle [radius=0.1];
 \draw[fill] (0,2) circle [radius=0.1];
 \draw[fill] (-1,1) circle [radius=0.1];
 \draw[fill] (1,1) circle [radius=0.1];
 \draw[thick] (0,2)--(1,3);
 \node[below] at (0,1.8) {\footnotesize{$rs$}};
\end{tikzpicture} $\quad +\quad$
\begin{tikzpicture}[baseline={(0,1)},scale=0.6]
 \draw[fill] (-1,3) circle [radius=0.1];
 \draw[fill] (1,3) circle [radius=0.1];
 \draw[fill] (0,2) circle [radius=0.1];
 \draw[fill] (-1,1) circle [radius=0.1];
 \draw[fill] (1,1) circle [radius=0.1];
 \draw[thick] (0,2)--(1,1);
 \node[below] at (0,1.8) {\footnotesize{$rs$}};
\end{tikzpicture} $\quad -\quad$
\begin{tikzpicture}[baseline={(0,1)},scale=0.6]
 \draw[fill] (-1,3) circle [radius=0.1];
 \draw[fill] (1,3) circle [radius=0.1];
 \draw[fill] (0,2) circle [radius=0.1];
 \draw[fill] (-1,1) circle [radius=0.1];
 \draw[fill] (1,1) circle [radius=0.1];
 \draw[thick] (0,2)--(1,3);
 \draw[thick] (0,2)--(1,1);
 \node[below] at (0,1.8) {\footnotesize{$rs$}};
\end{tikzpicture}
\end{center}
\caption{Six-term relations}\label{fig:6term}
\end{figure}

\begin{remark} The linear parts of these relations are exactly the linear relations in the Pukhlikov--Khovanskii ring of a Gelfand--Zetlin polytope, see~\cite[Prop.~3.2]{kirichenko2012schubert}. This corresponds to replacing the operator of shift by $e_{ij}$ by taking the directional derivative along the same vector.
\end{remark}

\begin{example} Let $n=3$. In this case, all ``six-term'' relations contain only four or two terms; they are as follows:
\begin{eqnarray*}
    [F_{11}]&=&[F^{11}]+[F_{21}]-[F_{11}^{21}],\\
    {}[F_{12}]+[F^{21}]-[F_{12}^{21}]&=&[F^{12}],\\
    {}[F_{21}]&=&[F^{21}].
\end{eqnarray*}
\end{example}

\section{Full flag varieties}\label{sec:flags}
In this section we apply the theory of discrete duality algebras to give a polyhedral description of $T$-equivariant $K$-theory  of full flag varieties $\GL(n)/B$ of type $A$. We further investigate our construction to give polyhedral presentations for ($T$-equivariant) structure sheaves of Schubert varieties in $\GL(n)/B$. We present a natural set of relations in $K_0(\GL(n)/B)$ coming from its polyhedral presentation. This relies on combinatorics of Gelfand--Zetlin polytopes and earlier results of \cite{kirichenko2012schubert}. We remark that a similar analysis can be carried out for Lagrangian full flags, that is, for $G/B$ with $G$ of type $C$, relying on results of \cite{Fujita22}, and possibly for other reductive groups using Littelmann's string polytopes. However, we leave the details to future work.

%In the last section we discuss a new approach to Schubert calculus on full flag varieties. It is based on the construction of Khovanskii--Pukhlikov ring, discussed in the previous section. We will mostly follow the paper \cite{kirichenko2012schubert}.

\subsection{Convex chain constructed by a permutation}\label{ssec:permut}

In this subsection, we shall assign to each Kogan face $F$ a word $\underline w(F)$ in the alphabet $s_1,\dots,s_{n-1}$ of Coxeter generators of the symmetric group $S_n$, as follows. We mark the edge going from $x_{i-1,j}$ to $x_{i,j}$ by a simple transposition $s_{i+j-1}\in S_n$ (recall that $1\leq i,j$ and $i+j\leq n$), as shown on Figure~\ref{fig:koganface}, and take the word in $s_1,\dots,s_{n-1}$ obtained by reading the letters on the edges \emph{from bottom to top from left to right}.

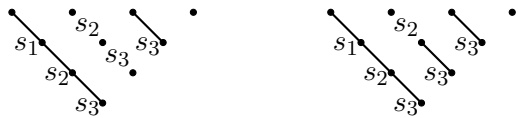
\begin{figure}[h!]
\begin{center}
 \begin{tikzpicture}[scale=0.4]
 \draw[fill] (-1,3) circle [radius=0.1];
 \draw[fill] (1,3) circle [radius=0.1];
 \draw[fill] (3,3) circle [radius=0.1];
 \draw[fill] (5,3) circle [radius=0.1];
 \draw[fill] (0,2) circle [radius=0.1];
 \draw[fill] (2,2) circle [radius=0.1];
 \draw[fill] (4,2) circle [radius=0.1];
 \draw[fill] (1,1) circle [radius=0.1];
 \draw[fill] (3,1) circle [radius=0.1];
 \draw[fill] (2,0) circle [radius=0.1];
 \draw[thick] (2,0)--(-1,3);
 \draw[thick] (3,3)--(4,2);
 \node[below] at (-0.5,2.5) {$s_1$};
 \node at (1.5,2.5) {$s_2$};
 \node[below] at (3.5,2.5) {$s_3$};
 \node[below] at (0.5,1.5) {$s_2$};
 \node at (2.5,1.5) {$s_3$};
 \node[below] at (1.5,0.5) {$s_3$};
 %\node[above] at (2.75,1.25) {$\footnotesize{s_1}$};
% \node at (2,0) {$F_2$};
\end{tikzpicture} \qquad\qquad
 \begin{tikzpicture}[scale=0.4]
 \draw[fill] (-1,3) circle [radius=0.1];
 \draw[fill] (1,3) circle [radius=0.1];
 \draw[fill] (3,3) circle [radius=0.1];
 \draw[fill] (5,3) circle [radius=0.1];
 \draw[fill] (0,2) circle [radius=0.1];
 \draw[fill] (2,2) circle [radius=0.1];
 \draw[fill] (4,2) circle [radius=0.1];
 \draw[fill] (1,1) circle [radius=0.1];
 \draw[fill] (3,1) circle [radius=0.1];
 \draw[fill] (2,0) circle [radius=0.1];
 \draw[thick] (2,0)--(-1,3);
 \draw[thick] (3,3)--(4,2);
 \draw[thick] (2,2)--(3,1);
 \node[below] at (-0.5,2.5) {$s_1$};
 \node at (1.5,2.5) {$s_2$};
 \node[below] at (3.5,2.5) {$s_3$};
 \node[below] at (0.5,1.5) {$s_2$};
 \node[below] at (2.5,1.5) {$s_3$};
 \node[below] at (1.5,0.5) {$s_3$};
 %\node[above] at (2.75,1.25) {$\footnotesize{s_1}$};
% \node at (2,0) {$F_2$};
\end{tikzpicture} 
\end{center}
\caption{Diagrams of Kogan faces}\label{fig:koganface}
\end{figure}

\begin{definition}
Let $\underline w=(s_{i_1},\dots,s_{i_k})$ be a word. The \emph{Demazure product} $\delta(\underline w)$ of $\underline w$ is the permutation defined inductively as follows: $\delta(s_i)=s_i$, and $\delta(\underline w,s_i)$ equals $\delta(\underline w)s_i$ if $\ell(\delta(\underline w)s_i)>\ell(\delta(\underline w))$, and $\delta(\underline w)$ otherwise. Note if $\underline w$ is a reduced word, then $\delta(\underline w)=s_{i_1}\dots s_{i_k}$.
\end{definition}

\begin{definition} Let $F$ be a Kogan face of codimension $k$, and let $\underline{w}(F)=(s_{i_1},\dots,s_{i_k})$ be the corresponding word. We shall say that $F$ \emph{corresponds} to the permutation $\delta(\underline w(F))$. A Kogan face is said to be \emph{reduced} if the word $\underline{w}(F)$ is reduced, and \emph{non-reduced} otherwise.
\end{definition}

\begin{example} 
Diagrams on Figure~\ref{fig:koganface} produce the words $(s_3,s_2,s_1,s_3)$ and  $(s_3,s_2,s_3,s_1,s_3)$ respectively.
 Both of these faces correspond to the permutation $s_3s_2s_1s_3=(4231)$. The left of them is reduced, while the right one is not. 
\end{example}
    
Given a permutation $w\in W$, denote by $\Fm(w)$ the set of all Kogan faces $F$ corresponding to $w$, and by $\Gamma(w)$ the (set-theoretic) \emph{union} of all these faces. 

\begin{example} Consider the two simple transpositions $s_1,s_2\in S_3$. For $s_1$, there exists only one Kogan face corresponding to it, namely, $F_{11}$ (see Figure~\ref{fig:GZpolytope}); it is defined by the equation $x=\lambda_3$.  For $s_2$, there will be three such faces. Two of them are reduced and have codimension 1; they are defined by the equations $y=\lambda_2$ and $x=z$ and denoted by $F_{12}$ and $F_{21}$ on the same figure. The third face is nonreduced and has codimension 2; it is the intersection of the first two, namely, $F_{12,21}$. So $\Fm(s_2)=\{F_{12},F_{21}, F_{12,21}\}$ and $\Gamma(s_2)=F_{12}\cup F_{21}$.
\end{example}

\subsection{Demazure modules and key polynomials}\label{ssec:key}

In this subsection, we give a definition of Demazure modules and recall a theorem by Kiritchenko, Smirnov, and Timorin relating their characters to faces of Gelfand--Zetlin polytopes.

%We start with some standard notation. Let $G=\GL(n)$. let $B$ be a fixed Borel subgroup in $G$, and let $T\subset B$ be the corresponding maximal torus. Denote by $A(T)$ the character ring of $T$.

Let $G=\GL(n)$. We further identify $W$ with $S_n$. For $w\in W$, let $X_w=\overline{BwB/B}\subset G/B$ be the corresponding Schubert variety; in particular, for the longest element $w_0$ we have $X_{w_0}=G/B$. Let $\calL(\lambda)$ be the $G$-equivariant line bundle on $G/B$ defined as $G\times^B \C_{-\lambda}\to G/B$, and let $\calL_w(\lambda)$ denote the restriction of $\calL(\lambda)$  to $X_w$.

\begin{definition} Let $V(\lambda)=H^0(G/B,\calL_\lambda)$ be the irreducible representation with the highest weight $\lambda$. A \emph{Demazure module} $V_w(\lambda)$ for $w\in S_n$ is defined as follows. Take a vector of weight $w\lambda$ in $V(\lambda)$ (such a vector is unique up to a scalar) and consider the $B$-module of $V(\lambda)$ generated by this vector. We denote this $B$-module by $V_w(\lambda)$.
\end{definition}

The following theorem is well-known (cf., for instance, \cite[Sec.~3.3]{BrionKumar05}).
\begin{theorem} For $\lambda\in \Lambda$, $w\in W$, we have a $B$-module isomorphism $H^0(X_w,\calL_w(\lambda))\cong V_w(\lambda)$.
\end{theorem}

The character of a Demazure module $\kappa_{w,\lambda}=\ch V_w(\lambda)$ is an element of the group algebra $R(T)=\Z[\Lambda]$ of the character group. It is called the \emph{key polynomial} corresponding to $w$ and $\lambda$. These polynomials, defined by M.\,Demazure in~\cite{Demazure74a, Demazure74b}, have many nice combinatorial properties and descriptions; for a survey, see, for example, the thesis of A.\,Pun~\cite{Pun16}.

The following theorem provides a relation between key polynomials and faces of Gelfand--Zetlin polytopes in terms of the collections of Kogan faces constructed from a Weyl group element as described in the previous subsection.
\begin{theorem}[{\cite[Theorem~5.1]{kirichenko2012schubert}}]\label{thm:Kogankey} The key polynomial $\kappa_{w,\lambda}$ is obtained as the sum over all integer points of the set of Kogan faces of $GZ(\lambda)$ corresponding to $w$:
\[
\kappa_{w,\lambda}=\sum_{x\in\Gamma(w)\cap \Z^N} e^{\pr(x)}.
\]
\end{theorem}

Specializing the character at $0$, we obtain the dimension formula for $V_w(\lambda)$:
\begin{corollary}\label{cor:dimDemazure} The Euler characteristic of $\calL_w(\lambda)$ is equal to 
\[
\chi(X_w, \calL_w(\lambda))=\dim V_w(\lambda)=\kappa_{w,\lambda}(0)=\#(\Gamma(w)\cap \Z^N).
\]
\end{corollary}

We finish this section by identifying classes of structure sheaves of Schubert varieties in $K_{GZ}$. In a view of Corollary~\ref{cor:dimDemazure}, we would like to say that the class $\Dm_w$ representing the class of the structure sheaf of a Schubert variety $X_w$ is given by
\[
\Dm_w = \sum_{\Gamma\in \Fm(w)}(-1)^{\ell(w)-\dim(\Gamma)}D_\Gamma,
\]
where $D_\Gamma$ is an element of $K_{GZ}$ defined by the identity
\[
D_\Gamma \cdot \Ehr (GZ_\lambda) = \#(\Gamma_\lambda\cap \Z^N).
\]
However, unlike in the case of $K$-ring associated to a smooth complete fan $\Sigma$ (see Lemma~\ref{lem:Dsig}), the operators $D_\Gamma$ do not necessarily exist in $K_{GZ}$ for any face $\Gamma$. Nevertheless, the whole sum $\Dm_w$ is a well-defined operator in $K_{GZ}$, as it represents a function in the cyclic module $\Sh(\lambda)\cdot \Ehr(GZ(
\lambda)) = \Sh(\lambda)\cdot \chi(G/B,\Lm(\lambda))$, see Remark~\ref{rem:homology}. Indeed,
\[
\Dm_w \cdot \Ehr(GZ(
\lambda)) = \#(\Gamma(w)\cap \Z^N) = \chi(X_w, \calL_w(\lambda)) \in \Sh(\Lambda)\cdot \chi(G/B,\Lm(\lambda)),
\]
where the first equality follows from the inclusion-exclusion formula. We arrive at the following statement.

\begin{theorem}\label{thm:schubert classes} For a permutation $w\in W$, the class $[\Om_w]$ of the structure sheaf of Schubert variety $X_w$ is represented in $K_{GZ}$ by $\Dm_w$. In particular, the class $\Dm_w$ is well defined in $K_{GZ}$. Moreover, the structure sheaf of Schubert variety $X_w$ has the same presentation in $T$-equivariant $K$-ring $K_{GZ}^T$.
\end{theorem}
\begin{proof}
Since $K_0(G/B)$ is generated  by $\Pic(G/B)$ it is enough to check that $\chi(\calL_w(\lambda))=\Dm_w\cdot \Ehr(GZ(\lambda))$. Therefore the theorem follows from Corollary~\ref{cor:dimDemazure} and the fact that $\Dm_w\cdot \Ehr(GZ(\lambda))=\#(\Gamma(w)\cap \Z^N)$ by the inclusion-exclusion formula. The equivariant statement follows directly from Theorem~\ref{thm:Kogankey}.
\end{proof}

\begin{remark}
    An alternative way to state Theorem~\ref{thm:schubert classes} is to consider a resolution of $\widetilde{GZ}$ of the Gelfand-Zetlin fan and to work in the module $M_{\widetilde{GZ},GZ}$ that is associated to this resolution (see Proposition~\ref{prop:extended} and the end of Subsection~\ref{ssec:structureK}). In the $K$-ring $K_{\widetilde{GZ}}$ there is a well-defined operator $D_\Gamma$ and thus  $\Dm_w$ is automatically well-defined. On the other hand, we have an epimorphism $\pi\colon K_{\widetilde{GZ}}\to M_{\widetilde{GZ},GZ}$ and an injection $i\colon K_{GZ}\to M_{\widetilde{GZ},GZ}$. An argument similar to the proof of Theorem~\ref{thm:schubert classes} shows that $\pi(\Dm_w)$ is in the image of $i$, so it defines a class in $K_{GZ}$. This approach, realized for the cohomology case in~\cite[Sec.~2]{kirichenko2012schubert}, is useful for doing concrete computations.
\end{remark}

\subsection{Example: computations in $K(\GL(3)/B)$}\label{ssec:example_gl3}

In this subsection we consider the first nontrivial example: the $K$-ring of three-dimensional flag variety $\GL(3)/B$. 

Let us take the three-dimensional Gelfand--Zetlin polytope $GZ(\lambda)$ with the first row given by a strictly increasing sequence $(\lambda_3,\lambda_2,\lambda_1)$. We also denote $a=\lambda_2-\lambda_3$ and $b=\lambda_1-\lambda_2$. This means that the edges $e_1$, $e_3$ and $e_5$  contain $a+1$ integer points each, while each of $e_2$, $e_2$, and $e_6$ contains $b+1$ integer points (see~Fig.\ref{fig:GZpolytope}).

The correspondence between collections of faces of $GZ(\lambda)$ and permutations was given in~\cite[Sec.~4.3]{kirichenko2012schubert}. It can be obtained by a direct computation. We give it in Table~\ref{tab:gz3}; for each $w\in S_3$, we provide  linear combination of the corresponding Kogan faces and their Ehrhart polynomial.

\begin{table}[]
    \centering
    \begin{tabular}{|c|c|c|}
        \hline
         $w$ & Faces & $\Ehr(w)$  \\
         \hline
         $s_1s_2s_1$ & $F_{11,12,21}$ &1\\
         $s_1s_2$ & $F_{11,12}$ & $a+1$\\
         $s_2s_1$ & $F_{11,21}$ & $b+1$\\
         $s_1$ &$F_{11}$ & $(b+1)(a+\frac b2+1)$\\
         $s_2$ &$F_{12}\cup F_{21}=F_{12}+F_{21}-F_{12,21}$ & $(a+1)(\frac{a}{2}+b+1)$\\
         $Id$ & $GZ(\lambda)$ & $\frac{1}{2}(a+1)(b+1)(a+b+2)$\\
         \hline
    \end{tabular}
    \caption{Combinations of faces and their Ehrhart polynomials}
    \label{tab:gz3}
\end{table}

Now let us perform some computations in $K(\GL(3)/B)$ using this table. First compute $[\Om_{s_1}][\Om_{s_2}]$. To do this, we take the elements of polytope ring corresponding to $s_1$ and $s_2$; these are $[F_{11}]$ and $[F_{12}]+[F_{21}]-[F_{12,21}]$. These convex chains are transversal, and their intersection is equal to $[F_{11,12}]+[F_{11,21}]-[F_{11,12,21}]$. This is the union of two edges adjacent to $F_{11}$ and the Kogan vertex. These edges correspond to the classes of structure sheaves $[\Om_{s_1s_2}]$ and $[\Om_{s_2s_1}]$ respectively, and the vertex corresponds to $[\Om_{s_1s_2s_1}]$. Hence we have
\[
[\Om_{s_1}][\Om_{s_2}]=[\Om_{s_1s_2}]+[\Om_{s_2s_1}]-[\Om_{s_1s_2s_1}].
\]

The same computation can be carried out in terms of Ehrhart polynomials. The Ehrhart polynomial for the convex chain $\I_{F_{11,12}}+\I_{F_{11,21}}-\I_{F_{11,12,21}}$ is equal to $a+b+1$. Now we need to represent this polynomial as a linear combination of $\Ehr(w)$'s:
\[
a+b+1=(a+1)+(b+1)-1,
\]
and we obtain the same result.

Similarly, we can compute $[\Om_{s_1}]^2$. This class corresponds to the face $F_{11}$. To compute its square, we need to replace this face by an equivalent convex chain representing faces transversal to it, by means of the relations from Sec.~\ref{ssec:relations}. We have $[F_{11}]=[F^{11}]+[F_{21}]-[F_{21}^{11}]$. The product of these two convex chains is
\[
[F_{11}]\cdot([F^{11}]+[F_{21}]-[F_{21}^{11}])=[F_{11,21}].
\]
(the intersection of $F_{11}$ with both $F^{11}$ and $F_{21}^{11}$ is empty). The edge $F_{11,21}$ corresponds to $s_2s_1$. So we have 
\[
[\Om_{s_1}]^2=[\Om_{s_2s_1}].
\]
Alternatively,  we can say that the Ehrhart polynomial of $F_{11,21}$ is equal to $b+1$, and so the corresponding class in the $K$-ring is $[\Om_{s_2s_1}]$.

\bibliographystyle{alpha}
\bibliography{tb}

\Addresses

\end{document}